\documentclass[11 pt]{amsart}
\usepackage[applemac]{inputenc}

\usepackage[dvipsnames,svgnames,x11names]{xcolor}
\definecolor{ffmblue}{HTML}{006092}
\usepackage[top=20mm,bottom=20mm,left=20mm,right=20mm]{geometry}

\usepackage[hyphens]{url}
\usepackage[colorlinks=true,linkcolor=Maroon,citecolor=Maroon,urlcolor=ffmblue,hypertexnames=false,linktocpage]{hyperref}
\usepackage{bookmark}
\usepackage{amsmath,thmtools,mathtools}
\mathtoolsset{showonlyrefs=true}
\newcounter{mparcnt}

\usepackage{fancyhdr}
\usepackage{esint}
\usepackage{enumerate}

\usepackage{pictexwd,dcpic}
\usepackage{graphicx}
\usepackage{caption}
\swapnumbers
\declaretheorem[name=Theorem,numberwithin=section]{thm}

\declaretheorem[name=Lemma,sibling=thm]{lemma}
\declaretheorem[name=Proposition,sibling=thm]{prop}
\declaretheorem[name=Definition,style=definition,sibling=thm]{defn}
\declaretheorem[name=Corollary,sibling=thm]{cor}
\declaretheorem[name=Assumption,style=definition,sibling=thm]{assum}

\declaretheorem[style=remark,name=Remark,numbered=no]{remark}

\numberwithin{equation}{section}

\newcommand{\ti}{\tilde}

\newcommand{\wh}{\widehat}
\newcommand{\bs}{\backslash}
\newcommand{\cn}{\colon}
\newcommand{\sub}{\subset}

\newcommand{\bbN}{\mathbb{N}}

\newcommand{\bbR}{\mathbb{R}}

\newcommand{\bbS}{\mathbb{S}}

\newcommand{\bbB}{\mathbb{B}}

\newcommand{\8}{\infty}

\newcommand{\al}{\alpha}

\newcommand{\de}{\delta}
\newcommand{\ep}{\epsilon}
\newcommand{\ka}{\kappa}

\newcommand{\Si}{\Sigma}

\newcommand{\De}{\Delta}
\newcommand{\Ga}{\Gamma}

\newcommand{\La}{\Lambda}

\newcommand{\cL}{\mathcal{L}}

\newcommand{\cR}{\mathcal{R}}

\newcommand{\cT}{\mathcal{T}}

\newcommand{\del}{\partial}
\newcommand{\n}{\nabla}

\newcommand{\fa}{\forall}
\newcommand{\rt}{\sqrt}

\newcommand{\ip}[2]{\left\langle #1,#2 \right\rangle}
\newcommand{\fr}[2]{\frac{#1}{#2}}
\newcommand{\tfr}[2]{\tfrac{#1}{#2}}
\newcommand{\x}{\times}

\DeclareMathOperator{\dist}{dist}

\newcommand{\pf}[1]{\begin{proof}#1 \end{proof}}
\newcommand{\eq}[1]{\begin{equation}\begin{alignedat}{2} #1 \end{alignedat}\end{equation}}

\newcommand{\br}[1]{\left(#1\right)}
\newcommand{\abs}[1]{\lvert #1\rvert}
\newcommand{\enum}[1]{\begin{enumerate}[(i)] #1 \end{enumerate}}

\newcommand{\Matrix}[1]{\begin{pmatrix} #1 \end{pmatrix}}

\newcommand{\ra}{\rightarrow}
\newcommand{\hra}{\hookrightarrow}

\newcommand{\mt}{\mapsto}

\newcommand{\mrm}{\mathrm}

\newcommand{\hp}{\hphantom}
\newcommand{\q}{\quad}

\begin{document}
\title[Quermassintegral inequalities for convex capillary hypersurfaces in the unit ball]{Capillary quermassintegral inequalities in the unit ball}
\begin{abstract}

This paper is about hypersurfaces with boundary lying in the Euclidean unit ball, which meet the unit sphere at a fixed angle $\theta\in(0,\fr{\pi}{2}]$. Such hypersurfaces are called $\theta$-capillary hypersurfaces and for those we introduce a new notion of convexity, which we call $\theta$-horocap-convexity. For such hypersurfaces, we prove the convergence of a curvature flow of Guan/Li type with capillary boundary. Remarkably, we prove this result for a class of curvature functions which include all quotients of symmetric polynomials and, as a consequence, we obtain the full set of quermassintegral inequalities in the $\theta$-horocap-convex case. In the strictly horocap-convex setting, we employ the flow to prove the geometric inequalities, while for the horocap-convex case and the characterization of the equality case, we develop new arguments which are interesting in their own right.

\end{abstract}
\date{\today}
\author{Shujing Pan}
\address{\flushleft\parbox{\linewidth}{{\bf Shujing Pan}\\Goethe-Universit\"at\\ Institut f\"ur Mathematik\\ Robert-Mayer-Str.~10\\ 60325 Frankfurt\\ Germany\\ {\href{mailto:pan@math.uni-frankfurt.de}{pan@math.uni-frankfurt.de}}}}

\author{Julian Scheuer}
\address{\flushleft\parbox{\linewidth}{{\bf Julian Scheuer}\\Goethe-Universit\"at\\ Institut f\"ur Mathematik\\ Robert-Mayer-Str.~10\\ 60325 Frankfurt\\ Germany\\ {\href{mailto:scheuer@math.uni-frankfurt.de}{scheuer@math.uni-frankfurt.de}}}}

\maketitle

\section{Introduction}

\subsection*{Results}

For $n\in \bbN$, we consider weakly convex hypersurfaces with boundary (i.e. the second fundamental form is weakly definite) of disk-type in the closed unit ball $\bbB^{n+1}$ with boundary the $n$-sphere $\bbS^{n} = \del\bbB^{n+1}$,
\eq{x\cn  (\mrm{int}(\bbB^{n}),\del\bbB^{n})\ra (\mrm{int}(\Si),\del\Si)\hra (\mrm{int}(\bbB^{n+1}),\bbS^{n}),}
which meet $\bbS^{n}$ at a fixed angle $\theta$. By convention, this means that two oriented orthonormal bases $(\bar N(x),\bar\nu(x))$ of the normal space $N_{x}\del\Si$ and $(\mu,\nu)$ of $N_{x}\del\Si$ are related via
\eq{\Matrix{\mu \\ \nu} = \Matrix{\sin\theta & \cos\theta \\ -\cos\theta & \sin\theta} \Matrix{\bar N \\ \bar\nu	},}
or equivalently
\eq{\Matrix{\bar N \\ \bar\nu} = \Matrix{\sin\theta & -\cos\theta \\ \cos\theta & \sin\theta} \Matrix{\mu \\ \nu	}.}
To fix those vectors uniquely, we choose $\bar N(x) = x$, the conormal $\mu$ of $\del\Si\sub\Si$ outward pointing and $\nu$, such that in the Gaussian formula for $\Si$,
\eq{\label{eq:Gauss formula}D_{X}Y = \n_{X}Y - h(X,Y)\nu,}
the second fundamental form $h$ is non-negative definite. If a surface $\Si$ meets the above configuration, in the following we call it {\it $\theta$-capillary}, or sometimes simply {\it capillary}.

This paper contains two main results. First, we consider a class of curvature flows $(\Si_{t})_{t>0}$ of $\theta$-capillary surfaces $\Si_{t}$, where $\theta\in (0,\pi/2]$. The flow speed is given by
\eq{\label{eq:flow}\dot{x} = \br{\fr{\ip{x+\cos\theta \nu}{e}}{F} - \ip{X_{e}}{\nu}}\nu + \cT,}
where $e$ is a suitable unit vector of $\bbR^{n+1}$ to be specified later, $X_{e}$ is the conformal Killing-field
\eq{\label{eq:Xe}X_{e}=\ip{x}{e}x-\tfr 12(\abs{x}^{2}+1)e,}
$F=F(\ka_{1},\dots,\ka_{n})$ lies in a certain class of functions of the principal curvatures $\ka_{1}\leq \dots\leq\ka_{n}$ and $\cT$ is a tangent part of the velocity which is required to preserve the condition $\del\Si_{t}\sub \bbS^{n}$. Notably, our class of $F$ does contain the curvature quotients $H_{k}/H_{k-1}$, where $H_{k}$ is the normalised $k$-th elementary symmetric polynomial of the principal curvatures. Due to the applicability of this quotient flow to geometric inequalities as discussed below, this choice of $F$ is of particular importance. The case $F=H_{n}/H_{n-1}$ and $\theta=\pi/2$ was treated in the paper \cite{ScheuerWangXia:02/2022} of the second author, Wang and Xia, where it was proven that the flow \eqref{eq:flow} with $\theta=\pi/2$, with strictly convex initial hypersurface, exists for all times and converges to a spherical cap, and in turn a new geometric inequality for quermassintegrals of free boundary (i.e. $\theta=\pi/2$) convex hypersurfaces was discovered. Due to possible failure of convexity preservation, in \cite{ScheuerWangXia:02/2022} it was not possible to generalise this to $k\neq n$, and hence also the geometric inequalities were left open, and indeed are open until today. The first result of our paper is the introduction of a new and more natural type of convexity, which capillary hypersurfaces can satisfy. As it turns out, this type of convexity is preserved under the flow \eqref{eq:flow} for a class of $F$ which contains the quotients $H_{k}/H_{k-1}$ for $1\leq k\leq n$ and in turn we can obtain the corresponding geometric inequalities. The $\theta$-capillarity for $\theta \in (0,\pi/2]$ does not add further problems here and hence we go beyond the free boundary case. As mentioned above, the key to all this is the introduction of a new type of convexity, which is defined as follows, where $g$ denotes the induced metric of $\Si$ and $h,\nu$ are as in \eqref{eq:Gauss formula}.

\begin{defn}\label{defn:horocap convexity}
For $\theta\in (0,\pi/2]$, a convex $\theta$-capillary $C^{2}$-hypersurface $\Si$ of $\bbB^{n+1}$ is called {\it $\theta$-horocap-convex}, if there exists $e\in \bbS^{n}$, which after rotation we may simply assume to be $e_{n+1}$, such that $\Si$ satisfies
\eq{\ip{x}{e}\geq \cos\theta\q\fa x\in \Si,}
and
\eq{(\ip{x-\cos\theta e}{e}) h\geq (1+\ip{e}{\nu})g.}
The hypersurface is called {\it strictly $\theta$-horocap-convex}, if  both inequalities are strict.
\end{defn}

We should mention, that the first condition is redundant in the sense that a convex $\theta$-capillary hypersurface always admits a vector with this property, see \cite[Prop.~2.16]{WengXia:10/2022} and \cite[Lemma~A.1]{HuWeiYangZhou:09/2023}. However, as the second condition also depends on that vector, it is more convenient to state this condition.

This convexity is motivated by the model case of certain spherical caps, which satisfy the equality in this definition, which we call {\it $\theta$-horocaps}, as they touch the hyperplane $\{x^{n+1}=\cos\theta\}$, see \autoref{sec:prelim} for further details. In this sense, it resembles the horosphere convexity of a hypersurface in hyperbolic space and this is why we chose this suggestive terminology.

Now we turn to the assumptions on the curvature function.
 
\begin{assum}\label{assum:F} 
Suppose that $\Ga\sub \bbR^{n}$ is a symmetric, open and convex cone containing the positive cone.
Let $F\in C^{\8}(\Ga)$ be a positive function satisfy:
\enum{
\item $F$ is strictly monotone increasing in each argument;
\item $F$ vanishes on $\del\Ga$;
\item $F$ is homogeneous of degree $1$ with $F(1,\dots,1)=1$;
\item $F$ is concave;
\item $F$ is inverse concave on $\Ga_+$, i.e. the function $\ka = (\ka_{i})\mt 1/F(\ka_{i}^{-1})$ is concave.
}
\end{assum}
\begin{remark}
 For the curvature quotients $F=\frac{H_{k}}{H_{k-1}}$, let $\Ga$ be the Garding cone $$\Ga_k=\{\ka\in\bbR^{n}|H_1>0,\cdots,H_k>0\},$$
 then $F$ satisfies all the conditions in \autoref{assum:F}.
\end{remark}
The first main result is the long-time existence and convergence of the flow \eqref{eq:flow} under the strict $\theta$-horocap-condition.

\begin{thm}\label{thm:flow}
Let $\theta\in (0,\pi/2]$ and $\Si_{0}\sub \bbB^{n+1}$ be strictly $\theta$-horocap-convex. Let $F$ satisfy \autoref{assum:F}. Then the flow \eqref{eq:flow} starting from $\Si_{0}$ exists for all times and converges smoothly to a spherical cap, which is rotationally symmetric around the $e_{n+1}$-axis.
\end{thm}

The main application is the full set of quermassintegral inequalities for $\theta$-horocap-convex capillary hypersurfaces of the unit ball.

\begin{thm}\label{thm:QM}
Let $\theta\in (0,\pi/2]$ and $\Si\sub \bbB^{n+1}$ be  $\theta$-horocap-convex. Then there holds
\eq{\label{thm:QM-A}W_{k}(\wh\Si)\geq f_{k}(W_{k-1}(\wh\Si)),\q 1\leq k\leq n,}
where $f_{k}$ is a strictly increasing function which gives equality on $\theta$-capillary spherical caps.
The equality case is precisely achieved on $\theta$-capillary spherical caps and flat balls.
\end{thm}
\begin{remark}
    The inequalities \eqref{thm:QM-A} for strictly $\theta$-horocap-convex capillary hypersurfaces follow directly from \autoref{thm:flow}. We then employ the capillary mean curvature flow to prove that a weakly $\theta$-horocap-convex capillary hypersurfaces can be approximated by a sequence of strictly $\theta$-horocap-convex capillary hypersurfaces (see \autoref{inequality} for details).
\end{remark}
The functionals $W_k$ are modifications of the classical quermassintegrals for closed hypersurfaces, adjusted by a term which balances the contribution coming from $\theta$-capillarity. It contains the quermassintegrals of $\del\Si\sub \bbS^{n}$ and takes the precise value of $\theta$ into account. The precise definitions were introduced in \cite[p.~8853]{WengXia:10/2022} as follows:
\begin{align*}
	W_{0,\theta}(\widehat{\Sigma})=|\widehat{\Sigma}|,\ 
	W_{1,\theta}(\widehat{\Sigma})=\tfrac{1}{n+1}\left(|\Sigma|-\cos\theta W_0^{\mathbb{S}^n}(\widehat{\partial\Sigma})\right),
\end{align*}
\begin{align*}
    W_{n+1,\theta}(\widehat{\Sigma})=\frac{1}{n+1}\left[\int_{\Sigma}H_{n}dA-\sum_{\ell=0}^{n-1}(-1)^{k+\ell}\binom{n}{\ell}\cos^{n-1-\ell}\theta\sin^\ell\theta W_\ell^{\mathbb{S}^n}(\widehat{\partial\Sigma})\right],
\end{align*}
and for $1\leq k\leq n-1$,
\begin{align*}
	\begin{aligned}
		&W_{k+1,\theta}(\widehat{\Sigma})=\frac{1}{n+1}\left\{\int_{\Sigma}H_{k}dA-\cos\theta\sin^k\theta W_k^{\mathbb{S}^n}(\widehat{\partial\Sigma})\right.  \\
		&\quad \quad \quad \left.-\sum_{\ell=0}^{k-1}\frac{(-1)^{k+\ell}}{n-\ell}\binom{k}{\ell}\left[(n-k)\cos^2\theta+k-\ell\right]\cos^{k-1-\ell}\theta\sin^\ell\theta W_\ell^{\mathbb{S}^n}(\widehat{\partial\Sigma})\right\},
	\end{aligned}
\end{align*}
where  $\widehat{\Sigma}$ is the domain in $\bbB^{n+1}$ enclosed by $\Si$, $\widehat{\partial\Sigma}$ is the convex body in $\bbS^{n}$ bounded by $\partial\Sigma$  and
$W_{k}^{\mathbb{S}^n}(\widehat{\partial \Sigma}), 0\leq k\leq n$, are the quermassintegrals of $\widehat{\partial \Sigma}$ in $\mathbb{S}^n$.
It is important to note that in the case $\theta = \pi/2$, the functionals reduce to those defined in \cite{ScheuerWangXia:02/2022}. Hence, \autoref{thm:QM} gives the first convexity condition under which one can deduce the full set of quermassintegral inequalities for free boundary hypersurfaces of the unit ball. The key property of the $W_{k}$ and the main reason for the precise form of their definition is that, under a variation of the form
\eq{\dot{x}^{\perp} = f\nu,}
their evolution is given by
\eq{\label{general variation}\del_{t}W_{k}(\wh\Si_{t}) = \fr{n+1-k}{n+1}\int_{\Si_{t}}H_{k}f,\q 0\leq k\leq n,}
as proven in \cite[Thm.~2.10]{WengXia:10/2022}.

\subsection*{Some history of the problem}

Inequalities for quermassintegrals of convex bodies, which in our formulation correspond to the case of closed convex hypersurfaces $\Si\sub \bbR^{n+1}$, are a classical topic in convex geometry. In this case, the functionals are given by
\eq{W_{k}(\wh\Si) = \fr{1}{n+1}\int_{\Si}H_{k-1},}
and their sharp inequalities are of the same form as in \eqref{thm:QM-A}, with the functions $f_{k}$ explicitly given by monomials with powers that ensure the scale invariance of the inequalities. The case $k=1$ is the isoperimetric inequality, and the case $k=2$ the Minkowski inequality. There are classical methods based on the Brunn-Minkowski theory to prove the inequalities in this setting, see the classical book on convex geometry by Schneider \cite{Schneider:/2014}. Over the course of the past decades, the use of curvature flows has become increasingly prominent for deriving generalizations of the quermassintegral inequalities in various directions, probably starting with the paper of Guan/Li \cite{GuanLi:08/2009}, in which the convexity assumption was relaxed to starshaped and $k$-convex (i.e. $H_{i}>0$ for all $i\leq k$). Modifications of the ambient space have also become possible, as seen in the works of Makowski and the second author \cite{MakowskiScheuer:11/2016} and Wei/Xiong \cite{WeiXiong:/2015}, which consider versions in the unit sphere, as well as in the work of  Wang/Xia \cite{WangXia:07/2014} for horo-convex hypersurfaces in hyperbolic space. Notably, in the closed and convex case, the full set of inequalities \eqref{thm:QM-A} is still unknown in both the sphere and hyperbolic space. We do not attempt to list other partial results here, but mention that the complete set of these inequalities was recently obtained by the authors in the class of horo-convex hypersurfaces of the unit sphere \cite{PanScheuer:12/2025}. These results also provide the foundation for the ideas developed in the present paper.
 
Instead, we focus on the case of locally constrained flows for hypersurfaces with boundary. Locally constrained flows of mean curvature type for closed hypersurfaces were first introduced by Guan/Li \cite{GuanLi:/2015} and read, for simplicity stated here only in the Euclidean space,
\eq{\label{eq:GL}\dot{x} = (1 - \ip{x}{\nu}H_{1})\nu,}
with many follow up developments such as those in \cite{GuanLiWang:/2019,LambertScheuer:03/2021,Scheuer:/2022,ScheuerXia:11/2019}. In recent years, the case of flows with free or capillary boundary conditions have received plenty of attention.

 It started with the free boundary flow in the unit ball along the evolution equation
 \eq{\dot{x} = (\ip{x}{e} - \ip{X_{e}}{\nu}H_{1})\nu}
 by Wang/Xia \cite{WangXia:/2022}, who proved smooth convergence to a spherical cap in the starshaped class.  See also \cite{MeiWeng:06/2023} for a space form version, and \cite{HuWeiYangZhou:09/2023,WangWeng:/2020} for the $\theta$-capillary version of this. Here $X_{e}$ is the conformal Killing field from \eqref{eq:Xe}. Subsequently, a fully nonlinear version of this flow was treated by the second author together with Wang and Xia \cite{ScheuerWangXia:02/2022}. The precise flow is
 \eq{\label{eq:SWX flow}\dot{x} = \br{\fr{\ip{x}{e}}{\fr{H_{n}}{H_{n-1}}}-\ip{X_{e}}{\nu}}\nu,}
 which is the special case $\theta = \pi/2$, also called the {\it free boundary} case, of the present paper with $F=H_n/H_{n-1}$. In \cite{ScheuerWangXia:02/2022}, we proved convergence of a convex free boundary initial hypersurface to a spherical cap, and thereby proved \eqref{thm:QM-A} in the case $k=n$. The restriction to $k=n$ stems from a major technical difficulty concerning the curvature flows: the multiplicative coupling of the curvature function $F = H_{k}/H_{k-1}$ with the term $\ip{x}{e}$ introduces a term in the evolution of the second fundamental form, for which 
no solution has yet been found to preserve convexity. Indeed it is, and so far remains, unknown whether  this flow preserves the convexity for general $F$, even in the free boundary case. In case $k=n$, this preservation comes for free. Similar issues arise for the corresponding flow of hypersurfaces in the unit sphere, even in the closed case. Indeed, the boundary conditions are not the hard part here, but the interior equation. There is one special situation in which a similar issue was overcome. In \cite{HuLiWei:04/2022}, Hu/Li/Wei proved the preservation of horo-convexity for the flow
 \eq{\dot{x} = \br{\fr{\cosh r}{F} - \ip{\sinh(r)\del_{r}}{\nu}}\nu,}
 where $r$ is the radial distance function from the closed hypersurface to the origin of the hyperbolic space. In turn, they proved that a stronger type of convexity is preserved and that the complicated term can be dealt with. With this result, they reproved the full set of quermassintegral inequalities of Wang/Xia \cite{WangXia:07/2014}, which were proved via a non-local flow in their paper. For the free or capillary boundary case, so far no one has been able to deal with flows which go beyond the choice $F=H_{n}/H_{n-1}$. The same issue forced Weng/Xia \cite{WengXia:10/2022} to use $F=H_{n}/H_{n-1}$ in the capillary version of \eqref{eq:SWX flow}. Let us also briefly mention the case of capillary surfaces in a half-space, which has undergone a similar development, see for example \cite{HuWeiYangZhou:10/2024,MeiWeng:01/2025,WangWengXia:02/2024,WangWengXia:08/2024}. However, note that concerning the above technical difficulty, the half-space case is less challenging, because here the multiplicative coupling involves the curvature function and the normal $\nu$, which yields an extra term in the evolution of the second fundamental form that does not directly hamper the application of the maximum principle. 
 
 Hence, in this case the set of inequalities is complete, and we focus on the case of the unit ball. Motivated by the approach of Hu/Li/Wei \cite{HuLiWei:04/2022}, we search for a new type of convexity and try to mimic the horo-sphere convexity of hypersurfaces in hyperbolic space. The outcome was what we called $\theta$-horocap-convexity in \autoref{defn:horocap convexity}. The key observation is that this convexity is preserved for a sufficiently large class of $F$, which allows us to deduce the full set of inequalities \eqref{thm:QM-A}. The rest of the paper is devoted to a detailed proof.

\subsection*{Funding}
This work was funded by the German Research Foundation (Deutsche Forschungsgemeinschaft,
DFG) through the project ``Curvature flows with local and nonlocal sources'', SCHE 1879/4-1.

\section{Preliminaries and horocap-convexity}\label{sec:prelim}

\subsection*{Notation}
In this section, we lay out the employed notation and conventions used in this paper. We consider smooth and convex hypersurfaces $\Si$ of the Euclidean unit ball $\bbB^{n+1}$ with metric $g$ and Levi-Civita connection $\n$ induced from the ambient Euclidean geometry with connection $D$. The second fundamental form $h$ and normal $\nu$ are chosen such that in the Gaussian formula \eqref{eq:Gauss formula} $h$ is non-negative definite. The hypersurface $\Si$, or also the time-dependent hypersurface $\Si_{t}$, will usually be embedded by maps $x(t,\cdot)$, which give rise to coordinate frames 
\eq{x_{i}:=\del_{i}x, \q 1\leq i\leq n.}
We use coordinate based notation, with common abbreviations
\eq{g_{ij} = g(x_{i},x_{j}),\q h_{ij}= h(x_{i},x_{j}),\q T_{ij;k} = (\n_{x_{k}}T)_{ij},}
for tensor fields $T$. We are raising and lowering indices with respect to the induced metric $g_{ij}$, e.g. $h^{i}_{j} = g^{ik}h_{kj}$, where $g^{ij}$ are the components of the inverse of $g$.

For our curvature function we use the standard notation, that it can be understood as a function of the Weingarten operator, $F=F(h^i_j)$, or as a function of the metric and the second fundamental form, $F=F(h_{ij},g_{ij})$. Then there holds, see \cite[Sec.~2.1]{Gerhardt:/2006},
\eq{F^{ij} := \fr{\del F}{\del h_{ij}} = \fr{\del F}{\del h^k_i}g^{kj}=:F^i_k g^{kj}.}

\subsection*{Horocap-convexity}
Our definition of horocap-convexity is geometrically motivated:

\begin{lemma}\label{horocap}
Let $\theta\in (0,\pi/2]$ and $\Si\sub \bbB^{n+1}$ be a $\theta$-capillary spherical cap with the properties
\eq{\fa x\in \Si\cn \ip{x}{e}\geq \cos\theta \q\mbox{and}\q \exists x_{0}\in \partial\Si\cn \ip{x_{0}}{e} = \cos\theta.} Then there holds
\eq{\ip{x-\cos\theta e}{e}h= (1+\ip{\nu}{e})g.}
\end{lemma}

\pf{
For a unit vector $e_0\in\bbR^{n+1}$, the $\theta$-capillary spherical cap of radius $R$ around $e_0$ is given by
\eq{C_{\theta,R}(e_0)=\{x\in\bbB^{n+1}: \vert x-\hat e_0\vert=R\}}
where $\hat{e}_0=\sqrt{R^2+2R\cos\theta+1}e_0$.
Since  $\ip{x}{e}\geq \cos\theta$ and there exists a boundary point $x_0$ such that $\ip{x_{0}}{e} = \cos\theta$, the identity $\ip{\hat{e}_0}{e} = \cos\theta+R$ follows. 
On all spherical caps, there holds $h_{ij} = \fr1R g_{ij}$ and hence 
\eq{(1+\ip{\nu}{e})g_{ij} = R\br{1+\ip{\fr{x-\hat e_0}{R}}{e}}h_{ij} = (\ip{x}{e}-\cos\theta)h_{ij}.}
 }

We also need a geometric lemma.

\begin{lemma}\label{lemma:support function estimate}
		For $\theta\in(0,\pi/2]$ and any convex $\theta$-capillary $C^2$-hypersurface $\Sigma\sub \bbB^{n+1}$, we have
\eq{\ip{x-\cos\theta e}{\nu}\leq 0.}
\end{lemma}
\begin{proof}
	 Along the boundary $\partial\Sigma$, the $\theta$-capillary boundary condition means $\langle x,\nu\rangle=-\cos\theta$. If $\langle x,\nu\rangle$ attains its maximum at some interior point $p\in\Sigma$, then 
	\begin{align*}
		\nabla_i\langle x,\nu\rangle_{p}=\kappa_i\langle x,e_i\rangle_p.
	\end{align*}
	If $\Sigma$ is strictly convex, $x\parallel\nu$ follows at $p$. Next, by
	\cite[Prop 2.16]{WengXia:10/2022}, there exists a point $e_0\in\widehat{\partial\Sigma}$ such that $\langle p,e_0\rangle\geq\cos\theta\geq0$ and $\langle\nu(p),e_0\rangle<0$. Therefore, $x/\abs{x}=-\nu$ holds, and hence
	\begin{align*}
		\langle p,\nu\rangle=-\vert p\vert\leq-\cos\theta.
	\end{align*}
	This shows that  $\langle x,\nu\rangle\leq -\cos\theta$ when $\Sigma$ is  strictly convex. An approximation argument (see \cite[Appendix A]{HuWeiYangZhou:09/2023}) then implies that the same inequality remains valid for a general convex hypersurface  $\Sigma$.
\end{proof}

\section{Preservation of horocap-convexity}

\subsection*{Short-time existence and the route to long-time existence}
The maximal time of smooth existence $T>0$ of the flow \eqref{eq:flow} with the given boundary condition is guaranteed by parabolic existence theory with oblique boundary conditions. To transfer the geometric problem into this setting, we  employ the conformal transformation utilised in \cite[Sec.~3.2]{WangWeng:/2020} from $\bbB^{n+1}$ to the open half space $\bbR^{n+1}_{+}$. For this purpose, we note that the initial hypersurface $\Si_0$ is strictly $\theta$-horocap-convex and in particular strictly convex. Hence we can find some $e_0\in \widehat{\del \Si_0}$, such that $\Si_0$ lies in the hemisphere determined by $e_0$ and is starshaped with respect to the conformal Killing field $X_{e_0}$. We can transform the unit ball conformally to $\bbR^{n+1}_{+}$, where $\Si_0$ is starshaped around the origin. Then, in \cite[Equ.~(3.5)]{WangWeng:/2020}, a scalar parabolic equation with oblique boundary condition is deduced:
	\begin{equation}\label{eq:scalar PDE}
		\left\{\begin{aligned}
			&\partial_tu=G(D^{2}u,Du,u,\cdot), \quad \text{in}\  \bbS^{n}_+\times(0,T),\\
			&\nabla_{-e_0}u=\cos\theta\sqrt{1+\vert D u\vert^2},  \quad \text{in}\  \partial\bbS^{n}_+\times[0,T),\\
            &u(0,\cdot) = u_0.
		\end{aligned}
		\right.
	\end{equation}
	The short-time existence in the class $C^{1+\fr{\al}{2},2+\al}([0,T)\x \bbB^n)\cap C^{\8}((0,T)\x \bbB^n)$ can be deduced from \cite[Thm.~14.23]{Lieberman:/1998}, also compare \cite[Thm.~6]{Uraltseva:06/1994}.
    
This also means that $F>0$ holds up to $T$, and in the following, we prove that the strict horocap-convexity with respect to some $e$ is preserved along the flow \eqref{eq:flow}. This requires that several properties have to be preserved along the flow: namely,  there is no interior touching of the flow with $x\in \bbS^{n}$, $\ip{x}{e}>\cos\theta$ and the condition
\eq{\cR:=\ip{x-\cos\theta e}{e}h - (1+\ip{e}{\nu})g>0.}

Once this is accomplished, we prove a lower bound on $F$, as well as a time dependent bound on the full second fundamental form. Long time existence will then follow from a short additional argument. To make this plan work,
 we need a few evolution equations, for which we define the operator
\eq{\cL = \del_{t} - \fr{\ip{x+\cos\theta\nu}{e}}{F^{2}}F^{ij}\n^{2}_{ij}-\ip{X_{e}-\fr{\cos\theta}{F}e+\cT}{\n}.}

\subsection*{Evolution equations}

\begin{lemma}
Along the flow \eqref{eq:flow}, there hold
\eq{\label{eq:ev height}\cL\ip{x}{e} = \fr{2\ip{x+\cos\theta\nu}{e}}{F}\ip{\nu}{e}+\fr{\cos\theta}{F}\abs{e^{\top}}^{2}-\ip{x}{e}^{2}+\tfr 12(\abs{x}^{2}+1),}
where ${}^{\top}$ denotes projection onto the tangent space of $\Si_{t}$,
\eq{\label{eq:ev dual height} \cL\ip{\nu}{e}= \fr{\ip{x+\cos\theta\nu}{e}}{F^{2}}F^{ij}h^{k}_{i}h_{kj}\ip{\nu}{e}- \fr{\abs{e^{\top}}^{2}}{F}+\ip{x}{\nu} - \ip{x}{e}\ip{\nu}{e},}
\eq{\del_t g_{ij}=\cL g_{ij}=2\br{\fr{\ip{x+\cos\theta\nu}{e}}{F}- \ip{X_{e}}{\nu}}h_{ij} + \cT_{j;i} + \cT_{i;j},}
\eq{\label{eq:ev h}\cL h_{ij}&= h^{k}_{j}\cT_{k;i} + h^{k}_{i}\cT_{k;j}+  \fr{\ip{x_{j} + \cos\theta h^{k}_{j}x_{k}}{e}}{F^{2}}F_{;i} + \fr{\langle x_{i} + \cos\theta h^{k}_{i}x_{k},{e}\rangle}{F^{2}}F_{;j}- \fr{2\ip{x+\cos\theta\nu}{e}}{F^{3}}F_{;i}F_{;j}\\
				&\hp{=}+ \fr{\ip{x+\cos\theta\nu}{e}}{F^{2}}F^{kl,pq}h_{kl;i}h_{pq;j}+ \fr{\ip{x+\cos\theta\nu}{e}}{F^{2}}F^{kl}h^{r}_{k}h_{rl}h_{ij}+\fr{\ip{\nu}{e}}{F}h_{ij} + \fr{\cos\theta \ip{\nu}{e}}{F}h^{k}_{i}h_{kj}   \\
				&\hp{=} + \ip{x}{e}h_{ij} - 2\ip{X_{e}}{\nu}h_{i}^{k}h_{kj} - \ip{e}{\nu}g_{ij}\\
				&\hp{=} + h^{k}_{j}(\ip{x_{i}}{e}\ip{x}{x_{k}}-\ip{x_{i}}{x}\ip{e}{x_{k}}) +  h^{k}_{i}(\ip{x_{j}}{e}\ip{x}{x_{k}}-\ip{x_{j}}{x}\ip{e}{x_{k}}).}
\end{lemma}

\pf{
Equation \eqref{eq:ev height} was proven in \cite[Prop.~3.5]{WengXia:10/2022}. For \eqref{eq:ev dual height} we use \cite[Prop.~2.11]{WengXia:10/2022}, as well as
\eq{\del_{t}\nu = \n\ip{X_{e}}{\nu} - \n\br{\fr{\ip{x+\cos\theta \nu}{e}}{F}} + h_{k}^{j}\cT^{k}x_{j},}

\eq{\label{eq:nu} \nu_{;i} = h^{k}_{i}x_{k},\q \nu_{;ij} = h^{k}_{i;j}x_{k} - h^{k}_{i}h_{kj}\nu,}
to compute
\eq{\cL\ip{\nu}{e} &= \ip{\n\ip{ X_{e}}{\nu}}{e} + \fr{\ip{x+\cos\theta\nu}{e}}{F^{2}}\ip{\n F}{e} - \fr{\abs{e^{\top}}^{2}}{F} - \fr{\cos\theta}{F}\ip{\n\ip{\nu}{e}}{e}\\
		&\hp{=} + h^{j}_{k}\cT^{k}\ip{x_{j}}{e} - \fr{\ip{x+\cos\theta\nu}{e}}{F^{2}}F^{ij}\ip{h^{k}_{i;j}x_{k} - h^{k}_{i}h_{kj}\nu}{e}-\ip{X_{e}-\fr{\cos\theta}{F}e+\cT}{\n \ip{\nu}{e}}\\
		&= \ip{\n\ip{ X_{e}}{\nu}}{e}  - \fr{\abs{e^{\top}}^{2}}{F} + \fr{\ip{x+\cos\theta\nu}{e}}{F^{2}}F^{ij}h^{k}_{i}h_{kj}\ip{\nu}{e}-\ip{X_{e}}{\n \ip{\nu}{e}}.}
The claimed equation follows from
\eq{\ip{\n\ip{X_{e}}{\nu}}{e} - \ip{X_{e}}{\n\ip{\nu}{e}} &= \ip{\n(\ip{x}{e}\ip{x}{\nu})}{e} - \tfr 12 \ip{\n((\abs{x}^{2}+1)\ip{\nu}{e})}{e}\\
				&\hp{=} - \ip{x}{e}\ip{x}{\n\ip{\nu}{e}} + \tfr 12 (\abs{x}^{2}+1)\ip{e}{\n\ip{\nu}{e}}\\
				&=\abs{e^{\top}}^{2}\ip{x}{\nu} + \ip{x}{e}h(x^{\top},e^{\top}) -  \ip{x}{e^{\top}}\ip{\nu}{e} - \ip{x}{e}\ip{x}{\n\ip{\nu}{e}} \\
				&=(1-\ip{e}{\nu}^{2})\ip{x}{\nu} -  \ip{x}{e-\ip{e}{\nu}\nu}\ip{\nu}{e}.
}
The evolution of the metric is standard and is deduced in \cite[Prop.~2.11]{WengXia:10/2022}. For the evolution of $h$, we start with
\eq{\label{ev-h-eq1}\del_{t}h_{ij} =& -\br{\fr{\ip{x+\cos\theta\nu}{e}}{F}- \ip{X_{e}}{\nu}}_{;ij} + \br{\fr{\ip{x+\cos\theta\nu}{e}}{F}- \ip{X_{e}}{\nu}}h_{ik}h^{k}_{j}\\
&+ h_{ij;k}\cT^{k}+h^{k}_{j}\cT_{k;i} + h^{k}_{i}\cT_{k;j},}
see \cite[Prop.~2.11]{WengXia:10/2022}. 
As usual, to get a formula for $\cL h$, we need to expand the Hessian term. This gives
\eq{-\br{\fr{\ip{x+\cos\theta\nu}{e}}{F}- \ip{X_{e}}{\nu}}_{;ij}&=-\fr{\ip{x_{;ij}+\cos\theta \nu_{;ij}}{e}}{F} + \fr{\ip{x_{j} + \cos\theta \nu_{;j}}{e}}{F^{2}}F_{;i} + \fr{\ip{x_{i} + \cos\theta\nu_{;i}}{e}}{F^{2}}F_{;j}\\
				&\hp{=} + \fr{\ip{x+\cos\theta\nu}{e}}{F^{2}}F_{;ij} - \fr{2\ip{x+\cos\theta\nu}{e}}{F^{3}}F_{;i}F_{;j} + \ip{X_{e}}{\nu}_{;ij}\\
				&=-\fr{\ip{-h_{ij}\nu+\cos\theta (h^{k}_{i;j}x_{k} - h^{k}_{i}h_{kj}\nu)}{e}}{F} + \fr{\ip{x_{j} + \cos\theta h^{k}_{j}x_{k}}{e}}{F^{2}}F_{;i}\\
				&\hp{=} + \fr{\ip{x_{i} + \cos\theta h^{k}_{i}x_{k}}{e}}{F^{2}}F_{;j}- \fr{2\ip{x+\cos\theta\nu}{e}}{F^{3}}F_{;i}F_{;j}\\
				&\hp{=} + \fr{\ip{x+\cos\theta\nu}{e}}{F^{2}}(F^{kl}h_{ij;kl} + F^{kl,pq}h_{kl;i}h_{pq;j} + F^{kl}h^{r}_{k}h_{rl}h_{ij} - F h^{k}_{i}h_{kj})\\
				&\hp{=} + h^{k}_{i;j}\ip{X_{e}}{x_{k}} + \ip{x}{e}h_{ij} - h_{i}^{k}h_{kj}\ip{X_{e}}{\nu} - g_{ij}\ip{e}{\nu}\\
				&\hp{=} + h^{k}_{j}(\ip{x_{i}}{e}\ip{x}{x_{k}}-\ip{x_{i}}{x}\ip{e}{x_{k}}) +  h^{k}_{i}(\ip{x_{j}}{e}\ip{x}{x_{k}}-\ip{x_{j}}{x}\ip{e}{x_{k}}), 
				}
where we also used \cite[Lemma~3.3]{ScheuerWangXia:02/2022} and the standard Simon's identity
\eq{F_{;ij} = F^{kl}h_{ij;kl} + F^{kl,pq}h_{kl;i}h_{pq;j} + F^{kl}h^{r}_{k}h_{rl}h_{ij} - F h^{k}_{i}h_{kj}.}
Putting things together gives
\eq{\cL h_{ij}&= h^{k}_{j}\cT_{k;i} + h^{k}_{i}\cT_{k;j}+  \fr{\ip{x_{j} + \cos\theta h^{k}_{j}x_{k}}{e}}{F^{2}}F_{;i} + \fr{\ip{x_{i} + \cos\theta h^{k}_{i}x_{k}}{e}}{F^{2}}F_{;j}- \fr{2\ip{x+\cos\theta\nu}{e}}{F^{3}}F_{;i}F_{;j}\\
				&\hp{=}+ \fr{\ip{x+\cos\theta\nu}{e}}{F^{2}}F^{kl,pq}h_{kl;i}h_{pq;j}+\fr{\ip{\nu}{e}}{F}h_{ij} + \fr{\cos\theta \ip{\nu}{e}}{F}h^{k}_{i}h_{kj}+ \fr{\ip{x+\cos\theta\nu}{e}}{F^{2}}F^{kl}h^{r}_{k}h_{rl}h_{ij}   \\
				&\hp{=} + \ip{x}{e}h_{ij} - 2\ip{X_{e}}{\nu}h_{i}^{k}h_{kj} - \ip{e}{\nu}g_{ij}\\
				&\hp{=} + h^{k}_{j}(\ip{x_{i}}{e}\ip{x}{x_{k}}-\ip{x_{i}}{x}\ip{e}{x_{k}}) +  h^{k}_{i}(\ip{x_{j}}{e}\ip{x}{x_{k}}-\ip{x_{j}}{x}\ip{e}{x_{k}}).}
}
The previous lemma provides all the necessary ingredients for computing the crucial tensor $\cR$.
\begin{lemma}
In coordinates, where $h_{ij}$ is diagonal and $g_{ij} = \de_{ij}$, along the flow \eqref{eq:flow}, the tensor $\cR$ evolves by
\eq{\cL\cR_{ii} &=  2\cR_{ki}\cT^{k}_{;i}+\fr{\ip{x+\cos\theta\nu}{e}}{F^{2}}F^{kl}h^{r}_{k}h_{rl}(\cR_{ii}+1)+ \fr{\cos\theta \ip{\nu}{e}}{F}h^{i}_{i}\cR_{ii} + \fr{\ip{\nu}{e}}{F}\cR_{ii}- \fr{1}{F}\cR_{ii} \\
		 &\hp{=}-\fr{\ip{x+\cos\theta\nu}{e}}{F}h_{ii}-2\ip{X_{e}}{\nu}h^{i}_{i}\cR_{ii}+ \tfr 12(\abs{x}^{2}+1-2\cos\theta\ip{x}{e})h_{ii}-\ip{x-\cos\theta e}{\nu}	\\
		&\hp{=} + \ip{x-\cos\theta e}{e}\br{  \fr{2\ip{x_{i} + \cos\theta h^{i}_{i}x_{i}}{e}}{F^{2}}F_{;i}  - \fr{2\ip{x+\cos\theta\nu}{e}}{F^{3}}(F_{;i})^{2}+ \fr{\ip{x+\cos\theta\nu}{e}}{F^{2}}F^{kl,pq}h_{kl;i}h_{pq;i}}\\
		&\hp{=} - \fr{2\ip{x+\cos\theta\nu}{e}}{F^{2}}F^{kl}\ip{x_{k}}{e} h_{ii;l}.	} 
\end{lemma}

\pf{
We compute
\eq{\cL\cR_{ij} &= \cL\ip{x}{e}h_{ij} + \ip{x-\cos\theta e}{e}\cL h_{ij} - \fr{2\ip{x+\cos\theta\nu}{e}}{F^{2}}F^{kl}\ip{x_{k}}{e} h_{ij;l}-\cL\ip{e}{\nu}g_{ij} - (1+\ip{e}{\nu})\cL g_{ij}\\
		&= \br{\fr{2\ip{x+\cos\theta\nu}{e}}{F}\ip{\nu}{e}+\fr{\cos\theta}{F}\abs{e^{\top}}^{2}-\ip{X_{e}}{e}}h_{ij}\\
		&\hp{=} + \ip{x-\cos\theta e}{e}\Big(  \fr{\ip{x_{j} + \cos\theta h^{k}_{j}x_{k}}{e}}{F^{2}}F_{;i} + \fr{\ip{x_{i} + \cos\theta h^{k}_{i}x_{k}}{e}}{F^{2}}F_{;j}- \fr{2\ip{x+\cos\theta\nu}{e}}{F^{3}}F_{;i}F_{;j}\\
				&\hp{=}+ \fr{\ip{x+\cos\theta\nu}{e}}{F^{2}}F^{kl,pq}h_{kl;i}h_{pq;j}+ \fr{\ip{x+\cos\theta\nu}{e}}{F^{2}}F^{kl}h^{r}_{k}h_{rl}h_{ij}+\fr{\ip{\nu}{e}}{F}h_{ij} + \fr{\cos\theta \ip{\nu}{e}}{F}h^{k}_{i}h_{kj}   \\
				&\hp{=} + \ip{x}{e}h_{ij} - 2\ip{X_{e}}{\nu}h_{i}^{k}h_{kj} - \ip{e}{\nu}g_{ij}+ h^{k}_{j}(\ip{x_{i}}{e}\ip{x}{x_{k}}-\ip{x_{i}}{x}\ip{e}{x_{k}})\\
				&\hp{=} +  h^{k}_{i}(\ip{x_{j}}{e}\ip{x}{x_{k}}-\ip{x_{j}}{x}\ip{e}{x_{k}})\Big)- \fr{2\ip{x+\cos\theta\nu}{e}}{F^{2}}F^{kl}\ip{x_{k}}{e} h_{ij;l}\\
		&\hp{=}-\br{\fr{\ip{x+\cos\theta\nu}{e}}{F^{2}}F^{kl}h^{r}_{k}h_{rl}\ip{\nu}{e}- \fr{\abs{e^{\top}}^{2}}{F}+\ip{x}{\nu} - \ip{x}{e}\ip{\nu}{e}}g_{ij}\\
		 &\hp{=}- 2(1+\ip{e}{\nu})\br{\fr{\ip{x+\cos\theta\nu}{e}}{F}- \ip{X_{e}}{\nu}}h_{ij} + \cR_{ki}\cT^{k}_{;j} + \cR_{kj}\cT^{k}_{;i}.}
Viewing these terms by order of $1/F$, one can observe, after some computations involving $\abs{e^\top}^2 = 1-\ip{e}{\nu}^2$,
 \eq{        
		 \cL \cR_{ij}&=  \cR_{ki}\cT^{k}_{;j} + \cR_{kj}\cT^{k}_{;i}+\fr{\ip{x+\cos\theta\nu}{e}}{F^{2}}F^{kl}h^{r}_{k}h_{rl}\cR_{ij} + \fr{\ip{x+\cos\theta\nu}{e}}{F^{2}}F^{kl}h^{r}_{k}h_{rl}g_{ij}\\
		 &\hp{=}+ \fr{\cos\theta \ip{\nu}{e}}{F}h^{k}_{i}\cR_{kj} + \fr{\ip{\nu}{e}}{F}\cR_{ij}- \fr{1}{F}\cR_{ij} -\fr{\ip{x+\cos\theta\nu}{e}}{F}h_{ij} -2\ip{X_{e}}{\nu}h^{k}_{i}\cR_{kj}\\
		 &\hp{=} + \tfr 12(\abs{x}^{2}+1)h_{ij} - \cos\theta\ip{x}{e}h_{ij}+(\cos\theta\ip{\nu}{e}-\ip{x}{\nu})g_{ij}	\\
		&\hp{=} + \ip{x-\cos\theta e}{e}\Big(  \fr{\ip{x_{j} + \cos\theta h^{k}_{j}x_{k}}{e}}{F^{2}}F_{;i} + \fr{\ip{x_{i} + \cos\theta h^{k}_{i}x_{k}}{e}}{F^{2}}F_{;j}- \fr{2\ip{x+\cos\theta\nu}{e}}{F^{3}}F_{;i}F_{;j}\\
				&\hp{=}+ \fr{\ip{x+\cos\theta\nu}{e}}{F^{2}}F^{kl,pq}h_{kl;i}h_{pq;j}+ h^{k}_{j}(\ip{x_{i}}{e}\ip{x}{x_{k}}-\ip{x_{i}}{x}\ip{e}{x_{k}}) \\
				&\hp{=}  +  h^{k}_{i}(\ip{x_{j}}{e}\ip{x}{x_{k}}-\ip{x_{j}}{x}\ip{e}{x_{k}})\Big)- \fr{2\ip{x+\cos\theta\nu}{e}}{F^{2}}F^{kl}\ip{x_{k}}{e} h_{ij;l}.}
The proof is complete.
}

\subsection*{Preservation of horocap-convexity}

We denote
\eq{T^\star :=\sup\{t<T\cn \Si_{t'}~\mbox{is strictly $\theta$-horocap-convex for all}\ t'\leq t\}}
 and suppose for the sake of contradiction that $T^{\star}<T$.
We first take care of the height estimate.

\begin{lemma}\label{height est}
Under the assumptions of \autoref{thm:flow}, along the flow \eqref{eq:flow} there holds
\eq{\fa t\in [0,T^{\star}]~\fa x\in \Si_{t}\cn\q \ip{x-\cos\theta e}{e}\geq  \de_{0}}
for some positive constant $\de_{0}$ which only depends on $\Si_0$. 
\end{lemma}

\pf{
Assume that $C_{\theta,R}(e)$ is a $\theta$-capillary spherical cap around the $e$-axis such that $\Si_0\subset\widehat{C_{\theta,R}(e)}$ and $\Si_0$ does not touch $C_{\theta,R}$ anywhere. Then $C_{\theta,R}(e)$ is a stationary solution to the flow \eqref{eq:flow}, see \cite[Equ.~(2.6)]{WengXia:10/2022}.
Define
\eq{d(t)=\inf\{\vert x(\cdot,t)-y\vert\cn (x(\cdot,t),y)\in\Si_t\times C_{\theta,R}(e)\}.}
Then $d(0)>0$ 
and we claim that $d$ remains positive. If $T^\star = 0$, there is nothing to prove.
Now, let $0<t<T^\star$ be the first time, where $d(t)=0$ and let $(x_t,y_t)\in\Si_t\times C_{\theta,R}(e)$ realize $d(t)$, i.e. $x_t = y_t$.
We first exclude the possibility that $x_t\in \del\Si_t$. If this were the case, then
\eq{
	\langle X_e,\nu\rangle(x_t)=-\langle e,\nu\rangle-\cos\theta\langle x_t,e\rangle=-\sin\theta\langle e,\bar{\nu}\rangle(x_t)>0,
}
because $x_t$ maximizes the spherical distance along $\del \Sigma_t$ to the north pole, and $\bar\nu$ is the outer normal to $\widehat{\del\Sigma_t}$, which then has to point downwards.
Hence, $\Sigma_t$ is  locally starshaped near the point $x_t$. This allows us to apply the argument  in \cite[Prop 4.2]{WangWeng:/2020} for starshaped surfaces to exclude this case via the Hopf boundary point lemma.
 Consequently, $x_t\in\text{int}(\Si_t)$, and the avoidance principle tells us that this case can not occur either. Hence $d(t)>0$ for all $t$, which shows that $\Sigma_t$ stays above $C_{\theta,R}(e)$. 
}

Now we need to keep the flow in the interior of the unit ball.

\begin{lemma}
Under the assumptions of \autoref{thm:flow}, along the flow \eqref{eq:flow} there holds $T^{\star}>0$ and 
\eq{\fa t\in [0,T^{\star}]~\fa x\in \mrm{int}(\Si_{t})\cn\q \abs{x}<1.}
\end{lemma}

\pf{
The positivity of $\cR$ and $\ip{x-\cos\theta e}{e}$ for a short time is trivial due to the compactness of $\Si_t$. Hence, for the first claim we only have to prove $\abs{x}<1$ is preserved for a short time. From the proof of \autoref{lemma:support function estimate} we obtain $\ip{x}{\nu}\leq -\cos\theta$ at $t=0$. This implies that for short time we have
\eq{\tfr 12\De\abs{x}^2 = -\ip{x}{\nu}H + n > 0. }
and hence $\abs{x}^2$ can not obtain any interior maxima for a short time.
To prove that $\Si_{T^\star}$ shares the desired property, we observe that $\Si_t$ is horocap-convex for all $t\leq T^\star$ and hence $\ip{x}{\nu}\leq 0$ continues to hold up to $T^\star$. The same argument as above implies that $\abs{x}^2$ can not attain any interior maximum on $\Si_{T^\star}$.
}

\begin{lemma}\label{lemma: R>0 pres}
	Assume that $\cR>0$ on the initial hypersurface $\Sigma_0$, then
the property $\cR>0$ is preserved up to and including $T^{\star}$. In particular we have $T^\star = T$.
\end{lemma}

\pf{
Let $\rho_{m} = \ip{x-\cos\theta e}{e}\ka_{m} - (1+\ip{e}{\nu})$, $m=1,\cdots,n$, be the eigenvalues of $\cR$ and $D$ be the multiplicity of $\rho_{1}$. Suppose the claim is false, then there is some point $\xi_0\in\Si_{T^\star}$ such that $\rho_{1}=0$ at $\xi_0$.

{\bf Case 1:} $\xi_{0}\notin \del\Si_{T^\star}$.
Then,
$\eta\equiv0$ is a smooth lower support of $\rho_1$ at $(T^\star,\xi_{0})$, i.e. $\eta= \rho_1(T^\star,\xi_{0})$ and  $\eta\leq \rho_1$. We also choose coordinates, such that $g_{ij} = \de_{ij}$ and $h_{ij}$ is diagonal.

Then by \cite[Lemma~4.1]{ChoiKimLee:12/2024}, at $(T^\star,\xi_{0})$ we obtain 
\eq{\label{gradient S}\cR_{ij;k} =0, \q \text{for} \quad 1 \le i,j  \le D,}
\eq{\label{maximum S}\del_{t}\cR^1_{1} \leq 0, \q \cR_{11;kk} - 2\sum_{l>D}\fr{(\cR_{1l;k})^{2}}{\rho_{l}-\rho_{1}}\geq0.}
We have to plug  the evolution of $\cR_{11}$ into   \eqref{maximum S}, and then, we can ignore all the terms of  gradients of $\cR_{11}$ by \eqref{gradient S}. We obtain, using $\cR_{11}=0$ at $(T^\star,\xi_{0})$
\eq{\label{pf:pres h-conv 3}0&\geq \fr{\ip{x+\cos\theta\nu}{e}}{F^{2}}F^{kl}h^{r}_{k}h_{rl}-\fr{\ip{x+\cos\theta\nu}{e}}{F}h_{11} + \tfr 12(\abs{x}^{2}+1-2\cos\theta\ip{x}{e})h_{11}-\ip{x-\cos\theta e}{\nu}	\\
		&\hp{=} + \ip{x-\cos\theta e}{e}\br{  \fr{2\ip{x_{1} + \cos\theta h^{1}_{1}x_{1}}{e}}{F^{2}}F_{;1}  - \fr{2\ip{x+\cos\theta\nu}{e}}{F^{3}}(F_{;1})^{2}+ \fr{\ip{x+\cos\theta\nu}{e}}{F^{2}}F^{kl,pq}h_{kl;1}h_{pq;1}}\\
		&\hp{=} - \fr{2\ip{x+\cos\theta\nu}{e}}{F^{2}}F^{kl}\ip{x_{k}}{e} h_{11;l} + \fr{2\ip{x+\cos\theta\nu}{e}}{F^{2}}F^{kk}\sum_{l>D}\fr{(\cR_{1l;k})^{2}}{\rho_{l}-\rho_{1}} \\
		&= \fr{\ip{x+\cos\theta\nu}{e}}{F^{2}}F^{kl}h^{r}_{k}h_{rl}-\fr{\ip{x+\cos\theta\nu}{e}}{F}h_{11} + \tfr 12(\abs{x}^{2}+1-2\cos\theta\ip{x}{e})h_{11}-\ip{x-\cos\theta e}{\nu}	\\
		&\hp{=} + \ip{x-\cos\theta e}{e}\br{  \fr{2\ip{x_{1} + \cos\theta h^{1}_{1}x_{1}}{e}}{F^{2}}F_{;1}  - \fr{2\ip{x+\cos\theta\nu}{e}}{F^{3}}(F_{;1})^{2}+ \fr{\ip{x+\cos\theta\nu}{e}}{F^{2}}F^{kl,pq}h_{kl;1}h_{pq;1}}\\
		&\hp{=} - \fr{2\ip{x+\cos\theta\nu}{e}}{F^{2}\ip{x-\cos\theta e}{e}}F^{kk}\ip{x_{k}}{e}^{2}(\ka_{k}-\ka_{1}) + \fr{2\ip{x+\cos\theta\nu}{e}\ip{x-\cos\theta e}{e}}{F^{2}}F^{kk}\sum_{l>D}\fr{(h_{1l;k})^{2}}{\ka_{l}-\ka_{1}},     }
where we used that for all $k,l,m$ we have
\eq{\label{pf:pres conv 1}\cR_{ml;k} = \ip{x_{k}}{e}h_{ml} + \ip{x-\cos\theta e}{e}h_{ml;k} - h^{r}_{k}\ip{e}{x_{r}}g_{ml} }
and the fact $\cR_{11;l} = 0$, which gives $\ip{x-\cos\theta e}{e}h_{11;l} = (\ka_{l}-\ka_{1})\ip{x_{l}}{e}$.	
We continue with
\eq{&\fr{2\ip{x+\cos\theta\nu}{e}\ip{x-\cos\theta e}{e}}{F^{2}}F^{kk}\sum_{l>D}\fr{(h_{1l;k})^{2}}{\ka_{l}-\ka_{1}}\\
=~&
\fr{2\ip{x+\cos\theta\nu}{e}\ip{x-\cos\theta e}{e}}{F^{2}}\sum_{k,l>D}F^{kk}\fr{(h_{1l;k})^{2}}{\ka_{l}-\ka_{1}}+ \fr{2\ip{x+\cos\theta\nu}{e}}{F^{2}\ip{x-\cos\theta e}{e}}\sum_{l>D}F^{11}(\ka_{l}-\ka_{1})\ip{x_{l}}{e}^{2},
}
which then yields
\eq{0&\geq  \fr{\ip{x+\cos\theta\nu}{e}}{F^{2}}F^{kl}h^{r}_{k}h_{rl}-\fr{\ip{x+\cos\theta\nu}{e}}{F}h_{11} + \tfr 12(\abs{x}^{2}+1-2\cos\theta\ip{x}{e})h_{11}-\ip{x-\cos\theta e}{\nu}	\\
		&\hp{=} + \ip{x-\cos\theta e}{e}\br{  \fr{2\ip{x_{1} + \cos\theta h^{1}_{1}x_{1}}{e}}{F^{2}}F_{;1}  - \fr{2\ip{x+\cos\theta\nu}{e}}{F^{3}}(F_{;1})^{2}+ \fr{\ip{x+\cos\theta\nu}{e}}{F^{2}}F^{kl,pq}h_{kl;1}h_{pq;1}}\\
		&\hp{=} + \fr{2\ip{x+\cos\theta\nu}{e}}{F^{2}\ip{x-\cos\theta e}{e}}(F^{11}-F^{ll})\ip{x_{l}}{e}^{2}(\ka_{l}-\ka_{1}) + \fr{2\ip{x+\cos\theta\nu}{e}\ip{x-\cos\theta e}{e}}{F^{2}}\sum_{k,l>D}\fr{F^{kk}(h_{1l;k})^{2}}{\ka_{l}-\ka_{1}}. }

Now we deal with the second derivatives of $F$. There holds, using $h_{kk;1} = 0$ for all $k\leq D$, \cite[Lemma~2.1.14]{Gerhardt:/2006}, and an estimate for inverse concave functions \cite[Thm.~2.1, Cor.~2.4]{Andrews:/2007},
\eq{F^{kl,pq}h_{kl;1}h_{pq;1}&=\fr{\del^{2}F}{\del \ka_{k}\del\ka_{l}}h_{kk;1}h_{ll;1} + 2\sum_{k>l}\fr{F^{kk}-F^{ll}}{\ka_{k}-\ka_{l}}(h_{kl;1})^{2}\\
						&=\sum_{k,l >D}\fr{\del^{2}F}{\del \ka_{k}\del\ka_{l}}h_{kk;1}h_{ll;1} + 2\sum_{l>D,k>l}\fr{F^{kk}-F^{ll}}{\ka_{k}-\ka_{l}}(h_{kl;1})^{2}+ 2\sum_{l\leq D,k>l}\fr{F^{kk}-F^{ll}}{\ka_{k}-\ka_{l}}(h_{kl;1})^{2}\\
						&=F^{kl,pq}\eta_{kl1}\eta_{pq1}+ 2\sum_{l\leq D,k>l}\fr{F^{kk}-F^{ll}}{\ka_{k}-\ka_{\ell}}(h_{kl;1})^{2}\\
						&\geq  \fr{2}{F}(F_{;1})^{2}-2\sum_{k,l>D}F^{kp}b^{l q}\eta_{kl1}\eta_{pq1}+ 2\sum_{l\leq D,k>l}\fr{F^{kk}-F^{ll}}{\ka_{k}-\ka_{l}}(h_{kl;1})^{2},}
where
\eq{\eta_{kl1} = \begin{cases} 0, &k\leq D~ \mbox{or}~ l\leq D\\
						h_{kl;1} & \mbox{else}.
			\end{cases} 
			}			
We obtain
\eq{0&\geq  \fr{\ip{x+\cos\theta\nu}{e}}{F^{2}}F^{kl}h^{r}_{k}h_{rl}-\fr{\ip{x+\cos\theta\nu}{e}}{F}h_{11} + \tfr 12(\abs{x}^{2}+1-2\cos\theta\ip{x}{e})h_{11}-\ip{x-\cos\theta e}{\nu}	\\
		&\hp{=} + \ip{x-\cos\theta e}{e}\left(  \fr{2\ip{x_{1} + \cos\theta h^{1}_{1}x_{1}}{e}}{F^{2}}F_{;1} + \fr{2\ip{x+\cos\theta\nu}{e}}{F^{2}}\sum_{l>D}\fr{\ka_{1}}{(\ka_{l}-\ka_{1})\ka_{l}}F^{ll}(h_{ll;1})^{2}\right)\\
		&= \fr{\ip{x+\cos\theta\nu}{e}}{F^{2}}F^{kl}h^{r}_{k}h_{rl}-\fr{\ip{x+\cos\theta\nu}{e}}{F}h_{11} +\tfr 12(\abs{x}^{2}+1-2\cos\theta\ip{x}{e})h_{11}-\ip{x-\cos\theta e}{\nu}	\\
		&\hp{=} +  \fr{2\ip{x+\cos\theta\nu}{e}\ip{x_{1}}{e}}{F^{2}}F_{;1} + \fr{2\ip{x-\cos\theta e}{e}\ip{x+\cos\theta\nu}{e}}{F^{2}}\sum_{l>D}\fr{\ka_{1}}{(\ka_{l}-\ka_{1})\ka_{l}}F^{ll}(h_{ll;1})^{2}, }
		where we also used the formula \eqref{pf:pres conv 1} to cancel the remaining term of the $F^{kl,pq}$ estimate.

Next, we deal with the remaining derivatives of curvature. If $\ka_1=0$, then $\ip{\nu}{e}=-1$ and $\ip{x_1}{e}=0$. Thus, 
the last two terms can be discarded. Otherwise,
we want to apply H\"older to control $\abs{F_{;1}}^{2}$ and therefore estimate
\eq{(F_{;1})^{2}\leq (\sum_{l>D}F^{ll}\abs{h_{ll;1}})^{2}&\leq \sum_{l>D}\fr{\ka_{1}}{(\ka_{l}-\ka_{1})\ka_{l}}F^{ll}(h_{ll;1})^{2} F^{kk}\fr{(\ka_{k}-\ka_{1})\ka_{k}}{\ka_{1}}, }
plug it in and obtain
\eq{0&\geq   \fr{\ip{x+\cos\theta\nu}{e}}{F^{2}}F^{kl}h^{r}_{k}h_{rl}-\fr{\ip{x+\cos\theta\nu}{e}}{F}h_{11} + \tfr 12(\abs{x}^{2}+1-2\cos\theta\ip{x}{e})h_{11}-\ip{x-\cos\theta e}{\nu}\\
		&\hp{=}+ \fr{\ip{x+\cos\theta\nu}{e}}{F^{2}}\left(  2\ip{x_{1}}{e}F_{;1} + 2\fr{1+\ip{\nu}{e}}{F^{kk}\ka_{k}^{2} - F\ka_{1}}(F_{;1})^{2}\right). }
Note that 
\eq{  2\ip{x_{1}}{e}F_{;1} + 2\fr{1+\ip{\nu}{e}}{F^{kk}\ka_{k}^{2} - F\ka_{1}}(F_{;1})^{2} \geq -\fr{F^{kk}\ka_{k}^{2}-F\ka_{1}}{2(1+\ip{\nu}{e})}\ip{x_{1}}{e}^{2}\geq -F^{kk}\ka_{k}^{2}+F\ka_{1},}
where we used $\langle x_1,e\rangle^2\leq 1-\langle\nu,e\rangle^2\leq 2(1+\langle\nu,e\rangle)$.
Hence we get
\eq{0&\geq  \tfr 12(\abs{x}^{2}+1- 2\cos\theta\ip{x}{e})\ka_{1}-\ip{x-\cos\theta e}{\nu}\\
	&= \tfr 12(\vert x-\cos\theta e\vert^2+1-\cos^2\theta)\ka_1-\ip{x-\cos\theta e}{\nu}.}
In case $\nu=-e$, we have $\kappa_1=0$ and 
$$-\ip{x-\cos\theta e}{\nu}=\langle x,e\rangle-\cos\theta\geq\delta_0>0$$ by \autoref{height est}, which yields a contradiction. Otherwise, we have $\ka_1>0$. Combining this with \autoref{lemma:support function estimate}, we again obtain a contradiction.

{\bf Case 2:} $\xi_{0}\in \del \Si_{T^\star}$. To exclude this case, using the notation $\ka_{\mu} = \ip{A(\mu)}{\mu}$, also see \cite[Prop.~2.5(4)]{WengXia:10/2022}, we calculate in case $x_{1}\perp\mu$,
\eq{\n_{\mu}\cR_{11}&=\ip{\mu}{e}(\ka_{1}-\ka_{\mu}) + \ip{x-\cos\theta e}{e}\n_{\mu}h_{11}\\
					&=\br{\ip{x-\cos\theta e}{e}\br{\cot\theta\ka_{1}+\tfr{1}{\sin\theta}}-\ip{\mu}{e}}(\ka_{\mu}-\ka_{1})\\
					&=\br{\ip{x-\cos\theta e}{e}\br{\cot\theta\ka_{1}+\tfr{1}{\sin\theta}}-\ip{\tfr{x}{\sin\theta} + \cot\theta\nu}{e}}(\ka_{\mu}-\ka_{1})\\
					&=\br{\cot\theta\ip{x-\cos\theta e}{e}\ka_{1}-\cot\theta(1+\ip{\nu}{e}}(\ka_{\mu}-\ka_{1})\\
					&=\cot\theta\cR_{11}(\ka_{\mu}-\ka_{1}).
			}
Hence this case is excluded by the Hopf-lemma. Otherwise, i.e. $\ka_{\mu}=\ka_{1}$, we use $\n_{\mu}F = 0$, see \cite[Prop.~3.6]{WengXia:10/2022}, and obtain for an orthonormal frame $(e_{\al})$ of $T\del\Si$,
\eq{\label{boundaryidentity}0=F^{\mu\mu}h_{\mu\mu;\mu} + F^{\al\al}h_{\al\al;\mu} = F^{\mu\mu}h_{\mu\mu;\mu} + F^{\al\al}(\cot\theta \ka_{\al} + \tfr{1}{\sin\theta})(\ka_{\mu}-\ka_{\al}),}
which implies $h_{\mu\mu;\mu}\geq 0$ and in turn $\cR_{\mu\mu;\mu}\geq 0$.

The conclusion $T^\star = T$ now follows from the fact that the set of times up to $T$, where $\Si_t$ is strictly horocap-convex, is open and closed, as we have proven in the three previous results. 
}

\section{Further curvature estimates}

To obtain regularity for the flow, we need control on the curvature function, the ellipticity constant and the full second fundamental form.
We start with the curvature function $F$.

\begin{lemma}
Along the flow \eqref{eq:flow} there holds
\eq{\label{eq:ev F}\cL F  &=  \fr{2F^{ij}}{F^2}\ip{x_{j} + \cos\theta h^{k}_{j}x_{k}}{e}F_{;i} - \fr{2\ip{x+\cos\theta\nu}{e}}{F^{3}}F^{ij}F_{;i}F_{;j}\\
		&\hp{=}+ \fr{\ip{x}{e}}{F}(F^{2}-F^{kl}h^{r}_{k}h_{rl})+(1-F^{ij}g_{ij})\ip{\nu}{e}. }
\end{lemma}

\pf{
Using \eqref{eq:ev h}, there holds
\eq{
\cL F&=F^{ij}\cL h_{ij}-\fr{\ip{x+\cos\theta\nu}{e}}{F^2}F^{ij}F^{kl,rs}h_{kl;i}h_{rs;j} + \fr{\del F}{\del g_{ij}}\del_t g_{ij}\\
&=2F^{ij}h^{k}_{j}\cT_{k;i}+  \fr{2F^{ij}}{F^2}\ip{x_{j} + \cos\theta h^{k}_{j}x_{k}}{e}F_{;i} - \fr{2\ip{x+\cos\theta\nu}{e}}{F^{3}}F^{ij}F_{;i}F_{;j}+ \fr{\ip{x+\cos\theta\nu}{e}}{F}F^{kl}h^{r}_{k}h_{rl}\\
		&\hp{=}+\ip{\nu}{e} + \fr{\cos\theta \ip{\nu}{e}}{F}F^{ij}h^{k}_{i}h_{kj}   + \ip{x}{e}F - 2\ip{X_{e}}{\nu}F^{ij}h_{i}^{k}h_{kj} - \ip{e}{\nu}F^{ij}g_{ij} + \fr{\del F}{\del g_{ij}}\del_{t}g_{ij}\\
		&=  \fr{2F^{ij}}{F^2}\ip{x_{j} + \cos\theta h^{k}_{j}x_{k}}{e}F_{;i} - \fr{2\ip{x+\cos\theta\nu}{e}}{F^{3}}F^{ij}F_{;i}F_{;j}\\
		&\hp{=}+ \fr{\ip{x}{e}}{F}(F^{2}-F^{kl}h^{r}_{k}h_{kl})+(1-F^{ij}g_{ij})\ip{\nu}{e},}
where we used
\eq{\fr{\del F}{\del g_{ij}} = -F^{ik}h^{j}_{k},}
see \cite[Equ.~(2.1.150)]{Gerhardt:/2006}.
}

Now we provide the lower bound for the curvature function.

\begin{prop}\label{prop:est-F}
Along the flow \eqref{eq:flow} starting from a strictly $\theta$-horocap-convex hypersurface, the function $F$ is bounded from below by a positive constant depending on initial data.
\end{prop}

\pf{
On the set 
\eq{G=\{(t,\xi)\in [0,T)\x \bbB^{n}\cn \ip{\nu}{e}\geq  -1/2\}}
there holds
\eq{F\geq \ka_{1}\geq \tfr14,}
due to the monotonicity of $F$ and the $\theta$-horocap-convexity. Hence, in order to bound $F$ from below, we may consider the function
\eq{w = \fr{\ip{x+\cos\theta \nu}{e}}{-\ip{\nu}{e}F}}
on $\ti G=([0,T)\x \bbB^{n})\bs G$,
and observe that an upper bound on $w$ implies a lower bound on $F$. Suppose that $w$ attains a spatial maximum at some $(t_{0},\xi_{0})\in[0,T)\x\del\bbB^{n}$, then from the $\theta$-horocap-convexity we get
\eq{\ka_{1}\geq \fr{1+\ip{\nu}{e}}{\ip{x-\cos\theta e}{e}} = \fr{1-\cos\theta \ip{x}{e} + \sin\theta\ip{\bar\nu}{e}}{\ip{x}{e}-\cos\theta}\geq \fr{1-\cos(\al-\theta)}{\ip{x}{e}-\cos\theta} \geq  \fr{1-\cos(\al_{0}-\theta)}{\ip{x}{e}-\cos\theta},}
where we used $\ip{x}{e}=:\cos\al\geq\cos \al_{0}>\cos\theta$ due to the presence of the initial barrier and
\eq{1 = \abs{e}^{2}\geq \ip{x}{e}^{2} + \ip{\bar \nu}{e}^{2}.}
Hence, on the boundary $w$ is uniformly bounded and we may restrict the investigation to spatial interior maxima. Therefore, we need the evolution equation of $w$ and it is computationally convenient to consider $\log w$ instead. Using \eqref{eq:ev height}, \eqref{eq:ev dual height} and \eqref{eq:ev F}, we have
\eq{\cL \log w &= \cL \log \ip{x+\cos\theta \nu}{e} - \cL\log(-\ip{\nu}{e}) - \cL\log F\\
			&=\fr{1}{F^{2}\ip{x+\cos\theta\nu}{e}}F^{ij}\ip{x_{;i}+\cos\theta \nu_{;i}}{e}\ip{x_{;j}+\cos\theta\nu_{;j}}{e} - \fr{\ip{x+\cos\theta\nu}{e}}{F^{2}\ip{\nu}{e}^{2}}F^{ij}\ip{\nu_{;i}}{e}\ip{\nu_{;j}}{e}\\
			&\hp{=} - \fr{\ip{x+\cos\theta\nu}{e}}{F^{4}}F^{ij}F_{;i}F_{;j} + \fr{\cL \ip{x}{e}+\cos\theta \cL\ip{\nu}{e}}{\ip{x+\cos\theta\nu}{e}}- \fr{\cL\ip{\nu}{e}}{\ip{\nu}{e}} - \fr{\cL F}{F}\\
			&= \fr{1}{\ip{x+\cos\theta\nu}{e}}\br{-\ip{X_{e}}{e}+\cos\theta (\ip{x}{\nu} - \ip{x}{e}\ip{\nu}{e})}+ \fr{1}{\ip{\nu}{e}}\br{ \fr{\abs{e^{\top}}^{2}}{F}-\ip{x}{\nu}}\\
			&\hp{=}+\fr{1}{F}(1+F^{ij}g_{ij})\ip{\nu}{e},
			  }
where we have already included obvious cancellations and also used 
\eq{\fr{F_{;i}}{F} =  \fr{\ip{x_{;i}+\cos\theta \nu_{;i}}{e}}{\ip{x+\cos\theta \nu}{e}} - \fr{\ip{\nu}{e}_{;i}}{\ip{\nu}{e}}, }
which holds at a spatial interior maximum of $w$. Due to strict negativity of $\ip{\nu}{e}$, the maximum principle implies the result.
} 

We continue with the second fundamental form.

\begin{prop}\label{prop:time kappa bound}
Along the flow \eqref{eq:flow} starting from a strictly $\theta$-horocap-convex hypersurface, the second fundamental form is bounded by a constant depending on initial data and $T$.
\end{prop}

\pf{
From the identity 
$$\cL h^{i}_{j} = g^{ik}\cL h_{kj}  - g^{im}\del_{t}g_{mr} h^{r}_{j},$$ we work in normal coordinates where $h$ is diagonal. Using \eqref{eq:ev h} and the fact  $\ip{x}{e}>\cos\theta$, we obtain
\eq{\cL h^{n}_{n}&= \fr{2}{F^{2}}(1+\cos\theta\ka_{n})\ip{x_{n}}{e}F_{;n} - \fr{2}{F^{3}}\ip{x+\cos\theta \nu}{e}(F_{;n})^{2} + \fr{\ip{x+\cos\theta \nu}{e}}{F^{2}}F^{kl,pq}h_{kl;n}h_{pq;n}\\
		&\hp{=}+\fr{\ip{x+\cos\theta\nu}{e}}{F^{2}}F^{kl}h^{r}_{k}h_{rl}\ka_{n}+\fr{\ip{\nu}{e}}{F}\ka_{n}-\fr{2}{F}\ip{x}{e}\ka_{n}^{2} - \fr{\cos\theta\ip{\nu}{e}}{F}\ka_{n}^{2} + \ip{x}{e}\ka_{n}-\ip{e}{\nu}\\
		&\leq \fr{(1+\cos\theta \ka_{n})^{2}\ip{x_{n}}{e}^{2}}{2F\ip{x+\cos\theta \nu}{e}}+\fr{\ip{x+\cos\theta\nu}{e}}{F^{2}}(F^{kl}h^{r}_{k}h_{rl}-F\ka_{n})\ka_{n}+\fr{\ip{\nu}{e}}{F}\ka_{n}-\fr{1}{F}\ip{x}{e}\ka_{n}^{2}\\
		&\hp{=} + \ip{x}{e}\ka_{n}-\ip{e}{\nu},}
where we also used the concavity of $F$.
Now
\eq{\label{pf:time kappa bound 1}\fr{\cos^{2}\theta\ip{x_{n}}{e}^{2}}{2\ip{x+\cos\theta\nu}{e}}\leq \fr{\cos^{2}\theta(1-\ip{\nu}{e}^{2})}{2\cos\theta(1+\ip{\nu}{e})}\leq \cos\theta}
and we obtain with the help of $\ip{x}{e}\geq \cos\theta + \al$ for some $\al>0$,
\eq{\cL h^{n}_{n}\leq -\fr{\al}{F}(\ka_{n}^{2} - c\ka_{n}-c) + \ka_{n} + 1, }
where $c$ depends only  on initial data. 
We obtain time-dependent bound on $\ka_{n}$, because 
at the boundary we have
\eq{\n_{\mu}h_{nn} = \br{\cot\theta\ka_{n}+\tfr{1}{\sin\theta}}(\ka_{\mu}-\ka_{n})} 
in case $x_{n}\neq \pm\mu$ and otherwise
\eq{h_{\mu\mu;\mu} = -\fr{F^{\al\al}{}h_{\al\al;\mu}}{F^{\mu\mu}}=-\fr{F^{\al\al}\br{\cot\theta\ka_{\al}+\tfr{1}{\sin\theta}}(\ka_{\mu}-\ka_{\al})}{{F^{\mu\mu}}} \leq 0.}
Hence $\ka_{n}$ will not attain large boundary maxima. 
}

\begin{cor}
The flow \eqref{eq:flow} starting from a strictly $\theta$-horocap-convex hypersurface exists smoothly for all times.
\end{cor}

\pf{
Suppose that $T<\8$.
From the height estimates we obtain
\eq{\ip{\nu}{e} + 1 \geq \de>0}
along the boundary and
for some constant $\de$. Employing the curvature estimates and the $\theta$-horocap-convexity, we see that the hypersurfaces $\del \Si_t\sub \bbS^n$ satisfy uniform upper and lower curvature bounds, see \cite[Prop.~2.5(2)]{WengXia:10/2022}. This translates to a uniform spherical inradius bound of $\del\Si_t\sub \bbS^n$.  Hence, provided $t$ is sufficiently close to $T$, there exists a center $e_0 \in \widehat{\del\Si_t}$ with a uniformly controlled distance to $\del\Si_t$, because the whole flow moves with bounded speed. Hence, all $\Si_t$ are uniformly starshaped with respect to the conformal Killing field $X_{e_0}$ and then, employing the scalar PDE \eqref{eq:scalar PDE}, we can extract a unique graphical limit, which can serve as a new initial hypersurface. Thus the flow can be extended beyond $T$. 
}

\begin{lemma}
    Let the initial hypersurface $\Si_0$ be strictly horocap-convex, then the flow becomes starshaped
around the north pole after some waiting time and then remains a starshaped flow.
\end{lemma}
\begin{proof}
 We argue with certain properties of the related flow
 \begin{align}\label{pf:starshaped 1}
     \dot{y} = \br{\fr{\ip{y+\cos\theta \nu}{e}}{\frac{H_n}{H_{n-1}}} - \ip{X_{e}}{\nu}}\nu + \cT,
 \end{align}
 which preserves $W_n$ and increases $W_{n-1}$. Now we start the flow $(N_{t})$ solving \eqref{pf:starshaped 1} from any $\theta$-capillary spherical cap lying above the hyperplane $\{x^{n+1}=\cos\theta\}$. Then, $N_0$ achieves equality in the inequality $W_{n}\geq f_n(W_{n-1})$, as established by Weng/Xia \cite[Theorem 1.3]{WengXia:10/2022}. Therefore, each subsequent $N_{t}$ must also be a $\theta$-capillary spherical cap with the same radius; otherwise, $W_{n-1}(N_t)$ would strictly increase while $W_n(N_t)$ remains constant. 
Hence, the flow $(N_{t})$ satisfies
 \begin{align*}
     \dot{y} = &\br{\fr{\ip{y+\cos\theta \nu}{e}}{\frac{H_n}{H_{n-1}}} - \ip{X_{e}}{\nu}}\nu + \cT\\
     =&\br{\fr{\ip{y+\cos\theta \nu}{e}}{F} - \ip{X_{e}}{\nu}}\nu + \cT,
 \end{align*}
because $\frac{H_n}{H_{n-1}} = F$ for spherical caps. Therefore, $(N_{t})$ also satisfies \eqref{eq:flow} and by uniqueness we have $M_{t} = N_{t}$. Consequently,  for \eqref{eq:flow} with an initial spherical cap, the flow hypersurfaces are spherical caps of the same radius. Due to the spherical barrier estimate, we know that the maximum distance from $\partial N_t$ to the north pole is decreasing. It follows that $N_t$ must converge to a spherical cap around the $e_{n+1}$-axis, since  the only stationary points of the flow \eqref{eq:flow} are such spherical caps. 

 Now let $\Si_{0}$ be a strictly horo-convex initial hypersurface and $(\Si_{t})_{t>0}$ be the flow \eqref{eq:flow} emerging from it. Consider a $\theta$-spherical cap lying inside $\widehat\Si_0$, it moves to a spherical cap  around the $e_{n+1}$-axis by the previous argument. From the avoidance principle, after some time the north pole must be in the region bounded by $\Si_{t}$, which ensures the starshapedness.
\end{proof}

In order to obtain time-independent curvature bounds, we modify the evolution of $h^{n}_{n}$ with a well-known test function. 
\begin{prop}
Along the flow \eqref{eq:flow} starting from a strictly $\theta$-horocap-convex hypersurface, the second fundamental form is bounded by a constant depending on initial data.
\end{prop}

\pf{
Since $e$  lies  in the interior of the spherically convex set bounded by $\del \Si_{t}\sub \bbS^{n}$ after a waiting time $t_{0}$, we have $\ip{\nu}{e}\leq -2\ep <0$ and $\abs{x-e}^{2}\geq \ep$ for some $\ep>0$. We define
\eq{z= \log h^{n}_{n} - \La\ip{x}{e},}
where $\La$ is a positive constant to be determined later.

Similarly to the proof of \autoref{prop:time kappa bound},  in suitable coordinates and at large maximal interior points, we obtain
\eq{\cL z &\leq -\fr{\al}{F}\ka_{n} + 1 - \fr{2\La\ip{x+\cos\theta\nu}{e}}{F}\ip{\nu}{e}-\La\fr{\cos\theta}{F}\abs{e^{\top}}^{2}+\La\ip{X_{e}}{e}+ \fr{\La^{2}\ip{x+\cos\theta\nu}{e}}{F^{2}}F^{kl}\ip{x_{k}}{e}\ip{x_{l}}{e}\\
		&\leq -\fr{1}{F}(\al\ka_{n} - c\La) + 1 -\fr{\La}{2}\abs{x-e}^{2}+ \fr{\La^{2}\ip{x+\cos\theta\nu}{e}}{F^{2}}F^{kl}\ip{x_{k}}{e}\ip{x_{l}}{e}. }
From the $\theta$-horocap-convexity, we have
\eq{(1+\ip{\nu}{e})F^{ij}g_{ij}\leq \ip{x-\cos\theta e}{e}F}
and hence
\eq{F^{kl}\ip{x_{;k}}{e}\ip{x_{;l}}{e}\leq F^{ij}g_{ij}(1-\ip{\nu}{e}^2)\leq (1-\ip{\nu}{e})\ip{x-\cos\theta e}{e} F\leq  2F. }
Hence, the term involving $F^{2}$ can be absorbed into the term involving $\al$, after choosing $\La$ large enough to eliminate $+1$.
This excludes large interior maxima.

At large boundary values of $\ka_{n}$ we have, due to $\ip{\mu}{e}>0$,
\eq{\n_{\mu}z = \br{\cot\theta\ka_{n}+\tfr{1}{\sin\theta}}(\ka_{\mu}-\ka_{n})\ka_{n}^{-1} - \La\ip{\mu}{e}<0,} 
in case $x_{n}\neq \pm\mu$ and otherwise
\eq{\n_{\mu}z = -\fr{F^{\al\al}{}h_{\al\al;\mu}}{\ka_{n}F^{\mu\mu}} - \La\ip{\mu}{e}=-\fr{F^{\al\al}\br{\cot\theta\ka_{\al}+\tfr{1}{\sin\theta}}(\ka_{\mu}-\ka_{\al})}{{\ka_{n}F^{\mu\mu}}} -\La\ip{\mu}{e} < 0.}
Hence $z$ will not attain large boundary maxima. 
}
	
\begin{cor}
	Let the initial hypersurface $\Sigma_{0}\sub \bbB^{n+1}$ be strictly $\theta$-horocap-convex. Then the flow \eqref{eq:flow} starting from $\Sigma_{0}$ exists for all times and converges smoothly to a spherical cap, which is rotationally symmetric around the $e_{n+1}$-axis.
\end{cor}
\begin{proof}
	In order to apply parabolic regularity, we come back to \eqref{eq:scalar PDE}.  Our geometric bounds yield uniform estimates up to $C^{2}$. By (ii) of \autoref{assum:F}, they also implies the uniform parabolicity of the curvature function $F$ along the evolution, since the principal curvatures remain within a compact subset of the cone $\Ga$. Standard parabolic regularity theory gives parabolic H\"older estimates for the second derivatives, and then applying Schauder theory to the linearisation of the operator gives estimates to arbitrary order.    Due to these uniform estimates, a diagonal argument together with the Arzela-Ascoli theorem implies that the sequence of flows 
	\eq{x_k(t,\cdot)=x(t+k,\cdot)}
	subsequentially converges to a limit flow $x_{\infty}$. Next, we show that any limit flow $x_{\infty}$ is a spherical cap $C_{\theta,r_{\infty}}(e)$ with fixed radius $r_{\infty}$, rotationally symmetric around the $e_{n+1}$-axis. By the concavity of $F$ we have $F\leq H_1$, and this implies that $W_{0}(\Sigma_t)$ is nondecreasing along the flow \eqref{eq:flow}. Indeed, 
	\begin{align*}
		\partial_tW_0(\wh{\Sigma}_t)&=\frac{n}{n+1}\int_{\Si_{t}}\br{\langle x+\cos\theta\nu,e\rangle\frac{1}{F}-\langle X_e,\nu\rangle}\geq0,
	\end{align*}
	where in the second equality we  have used the Heintze-Karcher type inequality proved in \cite[Thm.~1.1]{WangXia:11/2024}. Consequently, $W_0(\wh{\Sigma}_t)$ converges to a constant $c_0$ and along the limit flow, it attains this constant for all times. Hence along the limit flow we have $\del_t W_0 = 0$, the Heintze-Karcher inequality holds with equality and thus the limit flow is a flow of spherical caps. 

    Finally, by the same argument as in \cite[Prop 3.8]{ScheuerWangXia:02/2022} or \cite[Prop 3.15]{WengXia:10/2022}, we conclude that the $\theta$-capillary boundary spherical cap must be rotationally symmetric around the $e_{n+1}$-axis.
\end{proof}

\section{The proof of \autoref{thm:QM}}\label{inequality}

\subsection{Strictly horocap-convex case.}
Firstly, we use the flow result  \autoref{thm:flow} to prove the quermassintegral inequalities \eqref{thm:QM-A} for the stricty $\theta$-horocap-convex capillary hypersurfaces, the following proof is straightforward and well-known.

Assume that $\Si_0$ is a strictly $\theta$-horocap-convex capillary hypersurface, let $\Si_t$ be the solution to the flow \eqref{eq:flow} with $F=H_{k}/H_{k-1}$ starting from $\Si_0$ for $k\geq 1$.
We use the general variation formula \eqref{general variation}  and compute

\eq{\label{mon-1}\del_{t}W_{k}(\wh\Si_{t}) = \fr{n+1-k}{n+1}\int_{\Si_{t}}(H_{k-1}\ip{x+\cos\theta \nu}{e} - H_{k}\ip{X_{e}}{\nu}) = 0,}
which is a Hsiung-Minkowski type formula proved in \cite[Prop.~2.8]{WengXia:10/2022}.

Furthermore,  we obtain
\eq{\label{mon-2}\del_{t}W_{k-1}(\wh\Si_{t}) = \fr{n+2-k}{n+1}\int_{\Si_{t}}\br{\fr{H^2_{k-1}}{H_{k}}\ip{x+\cos\theta\nu}{e}-H_{k-1}\ip{X_{e}}{\nu}}\geq 0,}
where we used the Newton-Maclaurin inequalities
\eq{H_{k-2}H_{k}\leq H_{k-1}^{2}}
and again a Hsiung-Minkowski identity in case $k\geq2$ and the Heintze-Karcher  inequality in case $k=1$.
Hence, the flow \eqref{eq:flow} starting from $\Si_{0}=\Si$ with $F = H_{k}/H_{k-1}$ preserves $W_{k}$ and increases $W_{k-1}$ for $k\geq1$. As the flow converges to a spherical cap, at $t=\8$ the equality in \eqref{thm:QM-A} is attained, and therefore the inequality is evident for $\Si=\Si_{0}$.

\subsection{General horocap-convex case.}
Next, we use the result on strictly $\theta$-horocap-convex hypersurfaces in conjunction with an approximation argument to prove the quermassintegral inequalities in the $\theta$-horocap-convex case. The task is to approximate a $\theta$-horocap-convex hypersurface
by strictly $\theta$-horocap-convex hypersurfaces. As the capillarity and the convexity condition have to be respected in combination, this is a nontrivial task. In the case of usual convexity, an argument using mean curvature flow was given in \cite{LambertScheuer:09/2017} for the free boundary case and later in \cite[Thm.~A.3]{HuWeiYangZhou:09/2023} for the capillary case. This approach does not work for us, due to the more complicated convexity condition. In the following, we present a new argument.

We consider the mean curvature flow
\begin{equation}\label{MCF}
	\begin{cases}
		\partial_t x = -H\nu + V,
		& \text{in } {\bbB}^n \times [0,T), \\[0.4em]
		\langle \bar N \circ x , \nu \rangle = -\cos\theta,
		& \text{on } \partial {\bbB}^n \times [0,T),
	\end{cases}
\end{equation}
and denote $\Sigma_t = x({\bbB}^n,t)$. Here $\nu$ is the outward unit normal of $\Sigma_t$ and $V$ is the tangential component of $\partial_t x$, chosen so that
$V|_{\partial\Sigma_t}$ is tangent to $\partial\Sigma_t$.
\begin{lemma}\label{R-mcf}
	Along the flow \eqref{MCF},  there hold
 \eq{\partial_t\cR_{ij}
 	=\Delta\cR_{ij}+\cR_{ij;k}V^{k}+\cR_j^kV_{k;i}+\cR_i^kV_{k;j}-2\langle x^k,e\rangle h_{ij;k}+\vert h\vert^2\cR_{ij}-2Hh^k_i\cR_{kj}+\vert h\vert^2 g_{ij}}
and for every fixed $i$,
\eq{\partial_t\cR_{i}^i
 	=\Delta\cR^i_i+\cR^i_{i;k}V^{k} -2\langle x^k,e\rangle h^i_{i;k}+\vert h\vert^2\cR^i_{i}+\vert h\vert^2.  }
\end{lemma}
\begin{proof}
	First, using $\partial_t\langle x,e\rangle=-\langle H\nu,e\rangle+\langle V,e\rangle$ and $x_{;ij}=-h_{ij}\nu$, we obtain
	 \eq{\label{height-mcf}\partial_t\langle x,e\rangle=\Delta\langle x,e\rangle+\langle V,e\rangle.}
	Next,  by \eqref{eq:nu} and 
	\[\partial_t\nu=\nabla H+h^j_kV^kx_j,\] 
	we compute
	 \eq{\label{dual height-mcf}	\partial_t\langle\nu,e\rangle=&\Delta\langle\nu,e\rangle+\vert h\vert^2\langle\nu,e\rangle+h^{\ell}_kV^k\langle x_{\ell},e\rangle.}
	Moreover, combining the evolution equation
\eq{\partial_th_{ij}=H_{;ij}-Hh_{ik}h^k_j+ h_{ij;k}V^{k}+h^{k}_{j}V_{k;i} + h^{k}_{i}V_{k;j},
}
with Simon's identity
\eq{H_{;ij}=\Delta h_{ij}+\vert h\vert^2h_{ij}-Hh_{ik}h^k_j,}
we obtain 
\eq{\label{h-mcf}	\partial_th_{ij}=\Delta h_{ij}+\vert h\vert^2h_{ij}-2Hh^k_ih_{kj}+h_{ij;k}V^{k}+h_j^kV_{k;i}+h_i^kV_{k;j}.}
Now, combining \eqref{height-mcf},\eqref{dual height-mcf} with \eqref{h-mcf},
and using the definition of $\cR$ we get
\eq{
	\partial_t\cR_{ij}&=\partial_t\langle x,e\rangle h_{ij}+\langle x-\cos\theta e,e\rangle\partial_th_{ij}-\partial_t\langle\nu,e\rangle g_{ij}-(1+\langle\nu,e\rangle )\partial_tg_{ij}\\
	&=\Delta\cR_{ij}-2\langle x^k,e\rangle h_{ij;k}+\langle V,e\rangle h_{ij}-\vert h\vert^2\langle\nu,e\rangle g_{ij}-h^{\ell}_kV^k\langle x_{\ell},e\rangle g_{ij}\\
	&\hp{=}+\langle x-\cos\theta e,e\rangle\left(\vert h\vert^2h_{ij}-2Hh^k_ih_{kj}+h_{ij;k}V^{k} +h_j^kV_{k;i}+h_i^kV_{k;j}\right)\\
	&\hp{=}+(1+\langle\nu,e\rangle )(2Hh_{ij}-V_{j;i}-V_{i;j})\\
	&=\Delta\cR_{ij}+\cR_{ij;k}V^{k}+\cR_j^kV_{k;i}+\cR_i^kV_{k;j}-2\langle x_k,e\rangle h_{ij;k}+\vert h\vert^2\cR_{ij}-2Hh^k_i\cR_{kj}+\vert h\vert^2 g_{ij},
}
where we have used 
\eq{\label{eq:grad R}
	\cR_{ij;k}=\langle x_k,e\rangle h_{ij}+\langle x-\cos\theta e,e\rangle h_{ij;k}-h_k^{\ell}\langle x_{\ell},e\rangle g_{ij}.
}

For the mixed representation we compute
\eq{\del_t R^i_j = g^{ik}\del_t\cR_{kj} + (2Hh_{s}^i - g^{ir}V_{r;s} - g^{ir}V_{s;r})\cR^s_j}
and obtain the result.
\end{proof}

\begin{thm}\label{thm:MCF approx}
Let $\theta\in (0,\pi/2]$ and $\Si\sub\bbB^{n+1}$ be $\theta$-horocap-convex and non-flat. Then the $\theta$-capillary mean curvature flow $(\Si_t)_{t>0}$ converges in $C^{2,\al}$ to $\Si$ as $t\ra 0$. Moreover, all $(\Si_t)_{t>0}$ are strictly $\theta$-horocap-convex.
\end{thm}

\pf{
The short time existence in the class $C^{1+\fr{\alpha}{2},2+\alpha}([0,\de)\times\bbB^n)\cap C^{\infty}((0,\de)\times\bbB^n)$ of the flow \eqref{MCF}, starting from a convex initial data, can be guaranteed using arguments similar to those for the flow \eqref{eq:flow}, which are based on standard parabolic theory.

Next, we prove that the flow \eqref{MCF} preserves the $\theta$-horocap-convexity. Therefore
we define,
\eq{\ti \cR_{ij} = \cR_{ij} + \ep g_{ij}}
for arbitrary small positive constant $\ep$. Then $\ti\cR$ is positive at $t=0$, and we will prove that this remains valid over time. Then, we let $\ep\ra 0$. Let $\ti\rho_i$ be the eigenvalues of $\ti\cR$ and suppose there exists a first time $t_0>0$ and a point $\xi_0\in \Si_{t_0}$ such that $\ti\rho_1(t_0,\xi_0) = 0$.

{\bf Case 1:} $\xi_0\notin \del \Si_{t_0}$. Then $\eta \equiv 0$ is a smooth lower support of $\ti\rho_1$ at $(t_0,\xi_0)$. As in the proof of \autoref{lemma: R>0 pres} we obtain for the multiplicity $D$ of $\ti\rho_1(t_0,\xi_0)$,
\eq{\ti\cR_{ij;k} = \cR_{ij;k} = 0\q\fa 1\leq i,j \leq D}
and 
\eq{\del_t\ti\cR_1^1 \leq 0,\q \ti\cR_{11;kk}\geq 2\sum_{l>D} \fr{(\cR_{1l;k})^2}{\rho_l - \rho_1}.}
We insert the evolution of $\cR^1_1$ under mean curvature flow and obtain at $(t_0,\xi_0)$ using normal coordinates and \eqref{eq:grad R} twice,
\eq{0&\geq \Delta\cR^1_1+\cR^1_{1;k}V^{k}-2\langle x^k,e\rangle h^1_{1;k}+\vert h\vert^2\cR^1_{1}+\vert h\vert^2\\
    &\geq 2\sum_{l>D} \fr{(\cR_{1l;1})^2}{\rho_l - \rho_1}  -2\langle x^k,e\rangle h^1_{1;k}-\ep\vert h\vert^2+\vert h\vert^2\\
    &= 2\sum_{l>D} \fr{\ip{x-\cos\theta e}{e}^2(h_{11;l})^2}{\rho_l - \rho_1}  -2\langle x^k,e\rangle h^1_{1;k}+(1-\ep)\vert h\vert^2\\
    &= 2\sum_{l>D} \fr{\ip{x_l}{e}^2(\ka_l-\ka_1)^2}{\rho_l - \rho_1}  -2 \fr{\ip{x_k}{e}^2}{\ip{x-\cos\theta e}{e}}(\ka_k-\ka_1)+(1-\ep)\vert h\vert^2\\
    &=(1-\ep)\abs{h}^2,}
which is a contradiction, because we know from \cite[Thm.~A3]{HuWeiYangZhou:09/2023} that $\Si_{t_0}$ is strictly convex.

{\bf Case 2:} $\xi_0\in \del\Si_{t_0}$. This case is almost literally as in the proof of \autoref{lemma: R>0 pres}, except that we have to use $\n_\mu H = H >0$ to apply the Hopf lemma.

Hence we conclude that $\cR\geq 0$. Given this information, a word by word repetition of the above proof with $\ep=0$ remains valid to conclude the strict $\theta$-horocap-convexity for $t>0$. To see this, let $t_0>0$ and suppose that at some point $\xi_{t_0}\in \Si_{t_0}$ there holds $\rho_1(t_0,\xi_0)=0$. Then $(t_0,\xi_0)$ is a global minimum for $\rho_1$ and hence the above computation remains valid.  
}

\subsection{The equality case.} By the above approximation arguments, we have established the full set of quermassintegral inequalities \eqref{thm:QM-A} for $\theta$-horocap-convex hypersurfaces. It remains to characterize the equality case, and thereby finish the proof of \autoref{thm:QM}. 

 To achieve this, we need the following initial value independent estimate for the curvature function $F$ along the flow \eqref{eq:flow}.
 \begin{lemma}\label{lem:est-F}
 	Along the flow  \eqref{eq:flow} with  $F=\frac{H_{k}}{H_{k-1}}$, starting from a strictly $\theta$-horoconvex hypersurface, for any time $t\in(0,1]$ we have  the estimate:
	\eq{\frac{1}{F}\leq \frac{C}{\sqrt{t}},}
	where $C$ is a positive constant depending only on the initial barrier.
  \end{lemma}
\begin{proof}
		Since $\Sigma_t$ is $\theta$-horocap-convex, $F$ admits a uniform lower bound on the boundary $\partial\mathbb{B}^{n}$, depending only on the initial barrier. For interior points, 
	we consider the following auxiliary function:
	\[\gamma=\frac{\langle x+\cos\theta\nu,e\rangle}{(2-\vert x\vert^2)F}.\]
	Firstly, using \eqref{eq:ev height}, \eqref{eq:ev dual height} and \eqref{eq:ev F}, we obtain
	\begin{align}\label{eq:ev speed}
		\mathcal{L}\frac{\langle x+\cos\theta\nu,e\rangle}{F}&=\frac{\mathcal{L}\langle x+\cos\theta\nu,e\rangle}{F}-\frac{\langle x+\cos\theta\nu,e\rangle}{F^2}\mathcal{L}F+2\frac{\langle x+\cos\theta\nu,e\rangle}{F^3}F^{ij}F_{;i}\br{\frac{\langle x+\cos\theta\nu,e\rangle}{F}}_{;j}\nonumber\\
			&=\frac{\langle x+\cos\theta\nu,e\rangle}{F}\langle x,e\rangle\br{\frac{F^{kl}h^{r}_{k}h_{rl} }{F^2}-2}+\frac{\langle x+\cos\theta\nu,e\rangle}{F^2}\langle e,\nu\rangle\left(\dot{F}^{ij}g_{ij}+1\right)\nonumber\\
		&\hp{=}+\frac{1}{2F}(1+\vert x\vert^2+2\cos\theta\langle x,\nu\rangle)+\frac{\langle x+\cos\theta\nu,e\rangle}{F^3}F^{kl}h^{r}_{k}h_{rl}\cos\theta\langle\nu,e\rangle\\
		&\leq\frac{C_1}{F^2}+\frac{C_2}{F},
	\end{align}
	where we used the facts $F^2\leq F^{ij}h^{k}_{i}h_{kj}\leq (n-k+1)F^2$ and $1\leq F^{ij}g_{ij}\leq k$ for $F=\frac{H_k}{H_{k-1}}$. Here $C_1$ and $C_2$ are positive constants depending only on $n$ and $k$. Next, we compute
	\begin{align*}
		\partial_t\vert x\vert^2=2\br{\frac{\langle x+\cos\theta\nu,e\rangle}{F}-\langle X_e,\nu\rangle}\langle\nu,x\rangle+2\langle\cT,x\rangle,
	\end{align*}
	and
	\begin{align*}
		\frac{\langle x+\cos\theta\nu,e\rangle}{F^2}F^{ij}(\vert x\vert^2)_{;ij}=2\frac{\langle x+\cos\theta\nu,e\rangle}{F^2}F^{ij}g_{ij}-2\frac{\langle x+\cos\theta\nu,e\rangle}{F}\langle x,\nu\rangle.
	\end{align*}
	Combining the above, we have
	\begin{align}\label{eq:norm}
		\cL\vert x\vert^2=-2\frac{\langle x+\cos\theta\nu,e\rangle}{F^2}F^{i}_{i}+2\br{2\frac{\langle x+\cos\theta\nu,e\rangle}{F}-\langle X_e,\nu\rangle}\langle\nu,x\rangle-\left\langle X_e-\frac{\cos\theta}{F}e,\nabla\vert x\vert^2\right\rangle.
	\end{align}
	Now, combining \eqref{eq:ev speed} and \eqref{eq:norm}, we deduce
	\begin{align*}
		\cL\gamma&=\frac{1}{2-\vert x\vert^2}\cL	\frac{\langle x+\cos\theta\nu,e\rangle}{F}+\frac{\gamma}{2-\vert x\vert^2}\cL\vert x\vert^2+2\frac{\langle x+\cos\theta\nu,e\rangle}{(2-\vert x\vert^2)F^2}F^{ij}\gamma_{;i}(2-\vert x\vert^2)_{;j}\\
		&\leq2\frac{\langle x+\cos\theta\nu,e\rangle}{(2-\vert x\vert^2)F^2}F^{ij}\gamma_{;i}(2-\vert x\vert^2)_{;j}-\delta\gamma^3+C\gamma,
	\end{align*}
	where $\delta$ and $C$ are positive constants depending only on the initial barrier. It follows that, at any interior spatial maximum point, there holds
	\begin{align*}
		\cL(t\gamma^2)\leq-t\delta\gamma^4+Ct\gamma^2+\gamma^2.
	\end{align*}
	  Let $M = \sup_{\Sigma_t,\, t \in [0,1]} (t\gamma^2)$, and suppose it is attained at an interior point at time $t_0 > 0$, then the maximum principle yields $M\leq \frac{C+1}{\delta}$. This completes the proof.
\end{proof}

Now suppose that a $\theta$-horocap-convex hypersurface $\Sigma$ achieves equality in the inequality \eqref{thm:QM-A} and is not a flat ball. Let $\Sigma^\ep$ be a family of hypersurfaces evolving by the mean curvature flow with initial data $\Sigma$.  According to \autoref{thm:MCF approx}, the family $(\Sigma^\ep)_{\ep>0}$ consists of strictly $\theta$-horocap-convex hypersurfaces and converges to $\Sigma$ in $C^{2,\alpha}$ as $\ep \to 0$. For each $\ep > 0$, let $\Sigma_t^\ep$ denote the solution to the flow \eqref{eq:flow} with initial hypersurface $\Sigma^\ep$. By \autoref{lem:est-F}, for any fixed $t>0$, the hypersurface $\Sigma_t^\ep$ satisfies a uniform positive lower bound for $F$ that is independent of $\ep$. On the other hand, recalling the evolution equation of $h^n_n$ derived in the proof of \autoref{prop:time kappa bound}, we obtain a uniform $C^2$ estimate for $\Sigma_t^\ep$, independent of $\ep$. Due to these uniform estimates, for every $t>0$, standard parabolic regularity yields uniform $C^{2,\alpha}$ estimates, e.g. \cite[Thm.~14.22]{Lieberman:/1998}. Consequently, $\Sigma_t^\ep$ converges  to a hypersurface $\Sigma_t$ in $C^{2,\beta}(\beta<\alpha)$ as $\ep\to 0$, for any $t>0$. This yields a solution $(\Sigma_t)_{t>0}$ to the flow \eqref{eq:flow}.

Next, we conclude that $(\Sigma_t)_{t>0}$ must be a family of $\theta$-capillary spherical caps with fixed radius. Indeed, by the monotonicity formulas \eqref{mon-1} and \eqref{mon-2}, we have 
\begin{align*}
	W_{k}(\wh\Si^{\ep}_{t})=W_{k}(\wh\Si^{\ep}), \ \text{and}\ W_{k-1}(\wh\Si^{\ep}_{t})\geq W_{k-1}(\wh\Si^{\ep}).
\end{align*}
Passing to the limit as $\ep \to 0$, and using the $C^{2,\alpha}$ convergence of $\Sigma^\ep$ to $\Sigma$, for all $t>0$ we obtain
\begin{align*}
	W_{k}(\wh\Si_{t})= W_{k}(\wh\Si)=f_k(W_{k-1}(\wh\Si))\leq f_k(W_{k-1}(\wh\Si_t)).
\end{align*}
On the other hand, by the quermassintegral inequalities  \eqref{thm:QM-A} we established for the  $\theta$-horocap-convex hypersurfaces, we see that 
\begin{align*}
	W_{k}(\wh\Si_{t})\geq f_k(W_{k-1}(\wh\Si_t)).
\end{align*}
Combining the above inequalities yields
\begin{align*}
	W_{k}(\wh\Si_{t})\equiv f_k(W_{k-1}(\wh\Si_t)\equiv W_{k}(\wh\Si)
 \end{align*}
 for all $t>0$.
Applying the monotonicity formula \eqref{mon-2} once again, we conclude that $\wh\Si_{t}$ 
must be a $\theta$-capillary spherical cap with the fixed radius determined by $W_{k}(\wh\Si)$. 

Finally, we conclude that $\Sigma_t$  converges to $\Sigma$ in the Hausdorff sense as $t\to 0$.
It follows that $\Sigma$ is a $\theta$-capillary spherical cap.
Indeed,
\begin{align*}
	\text{dist}(\Sigma,\Sigma_t)&\leq\text{dist}(\Sigma,\Sigma^{\ep})+\text{dist}(\Sigma^{\ep},\Sigma^{\ep}_t)+\text{dist}(\Sigma^{\ep}_t,\Sigma_t)\\
	&\leq\text{dist}(\Sigma,\Sigma^{\ep})+C\rt{t}+\text{dist}(\Sigma^{\ep}_t,\Sigma_t),
\end{align*}
where we used \autoref{lem:est-F}, and $C$ is a positive constant independent of $\ep$. 
Hence
\eq{\dist(\Si,\Si_t)\leq \limsup_{\ep\ra 0}(\text{dist}(\Sigma,\Sigma^{\ep})+C\rt{t}+\text{dist}(\Sigma^{\ep}_t,\Sigma_t)) = C\rt{t}}
 Then, letting $t$ tends to $0$, we  complete the proof.

\begin{remark}
We would like to make the reader aware that for the characterisation of the equality case in the weakly horocap-convex situation, the standard argument from \cite{GuanLi:08/2009} is not available, because we do not know whether the inequality extends to an open class near the optimisor. This means, a general variation will leave the class within which we have information and hence we can not conclude the Euler-Lagrange equation. This issue was missed in the papers \cite{ChenSun:03/2022,HuWeiYangZhou:09/2023} and hence they lack a complete proof of the characterisation of the equality case in their respective inequalities.

Our method, though not directly applicable to their situation due to our different notion of convexity, still gives a general strategy, which might be adaptable to complete their argument.
\end{remark}

\providecommand{\bysame}{\leavevmode\hbox to3em{\hrulefill}\thinspace}
\providecommand{\MR}{\relax\ifhmode\unskip\space\fi MR }
\providecommand{\MRhref}[2]{%
  \href{http://www.ams.org/mathscinet-getitem?mr=#1}{#2}
}
\providecommand{\href}[2]{#2}

\end{document}